\documentclass[12pt,centertags,reqno]{amsart}

\usepackage[foot]{amsaddr}
\usepackage{latexsym}
\usepackage[english]{babel}
\usepackage[T1]{fontenc}
\usepackage[numbers]{natbib}
\usepackage{amssymb}
\usepackage{fancyhdr}
\usepackage{url}
\usepackage{hyperref}
\usepackage{verbatim}
\usepackage{leftidx}
\usepackage{bm}
\usepackage{color,graphicx}
\usepackage{tikz}
\usepackage{bbm}

\usepackage{mathrsfs, mathtools}
\usepackage{stmaryrd}

\usepackage{marvosym}

\numberwithin{equation}{section} \makeatletter
\renewcommand{\subsection}{\@startsection
{subsection}{2}{0mm}{\baselineskip}{-0.25cm}
{\normalfont\normalsize\bf}} \makeatother

\newtheorem{theorem}{Theorem}[section]
\newtheorem{lemma}[theorem]{Lemma}

\newtheorem{definition}[theorem]{Definition}
\newtheorem{remark}[theorem]{Remark}
\newtheorem{proposition}[theorem]{Proposition}
\newtheorem{example}[theorem]{Example}
\newtheorem{assumption}[theorem]{Assumption}

\def \E {\mathcal E}
\def \F {\mathcal F}

\def \H {\mathcal H}
\def \L {\mathcal L}
\def \M {\mathcal M}

\def \P {\mathbf P}
\def \Q {\mathbf Q}

\def \R {\mathbb R}
\def \bF {\mathbb F}

\def \bH {\mathbb H}
\def \bE {\mathbb E}

\newcommand{\ud}{\mathrm d}

\newcommand{\aC}[1]{{\color{red} #1}}

\DeclareMathOperator*{\essinf}{ess\,inf}
\DeclareMathOperator*{\esssup}{ess\,sup}
\DeclarePairedDelimiter{\abs}{\lvert}{\rvert}

\def \a {{(1)}}
\def \b {{(2)}}
\def \i {{(i)}}

\begin{document}

 \author[M.~Brachetta]{Matteo Brachetta}
\address{Matteo Brachetta, Department  of Mathematics,
Politecnico di Milano}\email{matteo.brachetta@polimi.it}

\author[G.~Giorgia]{Giorgia Callegaro}
\address{Giorgia Callegaro, Department of Mathematics, University of
Padova}\email{gcallega@math.unipd.it}

\author[C.~Ceci]{Claudia  Ceci}
\address{Claudia  Ceci, Department of Economic Studies,
University ``G. D'Annunzio'' of Chieti-Pescara, Viale Pindaro, 42,
I-65127 Pescara, Italy.}\email{c.ceci@unich.it}

\author[C.~Sgarra]{Carlo Sgarra}
\address{Carlo Sgarra, Corresponding Author, Department of Mathematics,
 Politecnico di Milano}\email{carlo.sgarra@polimi.it}

\title[Optimal reinsurance via BSDEs in a partially observable model with jump clusters]{Optimal reinsurance via BSDEs in a partially observable model with jump clusters}

\date{\today}

\begin{abstract}
We investigate the optimal reinsurance problem when the loss process exhibits jump clustering features and the insurance company has restricted information about the loss process. We maximize expected exponential utility of terminal wealth and show that an optimal solution exists. By exploiting both the Kushner-Stratonovich and Zakai approaches, we provide the equation governing the dynamics of the (infinite-dimensional) filter and characterize the solution of the stochastic optimization problem in terms of a BSDE, for which we prove existence and uniqueness of solution. After discussing the optimal strategy for a general reinsurance premium, we provide more explicit results in some relevant cases.
\end{abstract}

\maketitle





\section{Introduction}

Optimal reinsurance problems have attracted special attention during the past few years and they have been investigated in many different model settings. Insurance companies can hardly deal with all the different sources of risk in the real world, so they hedge against at least part of them, by re-insuring with other institutions. A reinsurance agreement allows the primary insurer to transfer part of the risk to another company and it is well known that this is an effective tool in risk management. Moreover, the subscription of such contracts is required by some financial regulators, see e.g. the Directive Solvency II in the European Union.
Large part of the existing literature focuses mainly on classical reinsurance contracts such as the proportional and the excess-of-loss, which were extensively investigated under a variety of optimization criteria, e.g. ruin probability minimization, dividend optimization and expected utility maximization. Here we are interested in the latter approach (see Irgens and Paulsen \cite{Irgens_Paulsen}, Mania and Santacroce \cite{Mania2010}, Brachetta and Ceci \cite{Brachetta_Ceci_2019} and references therein).
Some of the classical papers devoted to the subject assume a diffusive dynamics for the surplus process, while the more recent literature considers surplus processes including jumps.

The pioneering risk model with jumps in non-life insurance is the classical Cram\'er-Lundberg model, where the claims arrival process is a Poisson process with constant intensity. This assumption implies that the instantaneous probability that an accident occurs is always constant, which is in a way too restrictive in the real world, as already motivated by Grandell \cite{Grandell}.
In recent years, many authors made a great effort to go beyond the classical model formulation. For example, Cox processes were employed to introduce a stochastic intensity for the claims arrival process, see e.g. Albrecher and Asmussen \cite{Albrecher_Asmussen}, Bjork and Grandell \cite{Bjork_Grandell}, Embrechts et al. \cite{Embrechts_Schmidli_Grandell}. Moreover, other authors introduced Hawkes processes in order to capture the self-exciting property of the insurance risk model in presence of catastrophic events. Hawkes processes were introduced by Hawkes \cite{Hawkes_1971} to describe geological phenomena with clustering features like earthquakes. Hawkes processes with general kernels are not Markov processes: they can eventually include long-range dependence, while Hawkes processes with exponential kernel exhibit the appealing property that the couple process-intensity is Markovian; moreover they are affine processes according to the definition provided by Duffie, Filipovic and Schachermayer \cite{Duffie_Filipovic_Schachermayer}.
For the latter literature strand here we mention Stabile and Torrisi \cite{Stabile_Torrisi} and Swishchuk et al. \cite{Swishchuk}.

Dassios and Zhao \cite{Dassios_Zhao_2011} proposed a model which combines the two approaches by introducing a Cox process with shot noise intensity and a Hawkes process with exponential kernel for describing the claim arrival dynamics.
Recently Cao, Landriault and Li \cite{Cao_Landriault_Li} investigated the optimal reinsurance-investment problem in the model setting proposed by Dassios and Zhao \cite{Dassios_Zhao_2011} with a reward function of mean-variance type.

A different line of research related to the optimal-reinsurance investment problem focuses on the possibility that the insurer does not have access to all the information when choosing the reinsurance strategy. As a matter of fact, only the claims arrival and the corresponding disbursements are observable. In this case we need to solve a stochastic optimization problem under partial information. Liang and Bayraktar \cite{Liang_Bayraktar} were the first to introduce a partial information framework in optimal reinsurance problems. They consider the optimal reinsurance and investment problem in an unobservable Markov-modulated compound Poisson risk model, where the intensity and jump size distribution are not known, but have to be inferred from the observations of claim arrivals.
Ceci, Colaneri and Cretarola \cite{Ceci_Colaneri_Cretarola} derive risk-minimizing investment strategies when information available to investors is restricted and they provide optimal hedging strategies for unit-linked life insurance contracts.
Jang, Kim and Lee \cite{Jang_Kim_Lee} present a systematic comparison between optimal reinsurance strategies in complete and partial information framework and quantify the information value in a diffusion setting.

More recently, Brachetta and Ceci \cite{Brachetta_Ceci_2020} investigate the optimal reinsurance problem under the criterion of maximizing the expected exponential utility of terminal wealth when the insurance company has restricted information on the loss process in a model with claim arrival intensity and claim sizes distribution affected by an unobservable environmental stochastic factor. 

In the present paper we investigate the optimal reinsurance strategy for a risk model with jump clustering properties in a partial information setting. The risk model is similar to that proposed by Dassios and Zhao \cite{Dassios_Zhao_2011} and it includes two different jump processes driving the claims arrivals: one process with constant intensity describing the exogenous jumps and another with stochastic intensity representing the endogenous jumps, that exhibits self-exciting features. The externally-excited component represents catastrophic events, which generate claims clustering increasing the claim arrival intensity. The endogenous part allows us to capture the clustering effect due to self-exciting features. That is, when an accident occurs, it increases the likelihood of such events.
The insurance company has only partial information at disposal, more precisely the insurer can only observe the cumulative claims process. The externally-excited component of the intensity is not observable and the insurer needs to estimate the stochastic intensity by solving a filtering problem. Our approach is substantially different from that of Cao et Al. \cite{Cao_Landriault_Li} in several respects: firstly, we work in a partial information setting; secondly, the intensity of the self-excited claims arrival exhibits a slight more general dependence on the claims severity; finally, we maximize an exponential utility function instead of following a mean-variance criterion. In a partially observable framework, our goal is to characterize the value process and the optimal strategy. The optimal stochastic control problem in our case turns out to be infinite dimensional and the characterization of the optimal strategy cannot be performed by solving a Hamilton-Jacobi-Bellman equation, but via a BSDE approach.

A difficulty naturally arises when dealing with Hawkes processes: the intensity of the jumps is not bounded a priori, although a non-explosive condition holds. Hence we are not able to exploit some relevant bounds, which are usually required to prove a verification theorem and results on existence and uniqueness of the solution for the related BSDE. Nevertheless, we are going to show that the optimal stochastic control problem has a solution, which admits a characterization in terms of a unique solution to a suitable BSDE.

Our paper aims to contribute in different directions to the literature on optimal reinsurance problems: first, we provide a rigorous and formal construction of the dynamic contagion model. Second, we study the filtering problem associated to our problem, providing a characterization of the filter process in terms of the Kushner-Stratonovich equation and the Zakai equation as well. To the best of our knowledge, this problem has not been addressed in the existing literature. We refer to Dassios and Jang \cite{Dassios_Jang} for a similar problem without the self-exciting component. Third, we solve the optimal reinsurance problem under the expected utility criterion.

We remark that our study differs from Brachetta and Ceci \cite{Brachetta_Ceci_2020} in many key aspects. The risk model is substantially different, requires a strong effort to be rigorously constructed and the study of a new filtering problem. What is more, a crucial assumption in Brachetta and Ceci \cite{Brachetta_Ceci_2020} is the boundedness of the claims arrival intensity, which is not satisfied in our case, thus leading to additional technicalities in most of the proofs. This is what happens, for example, to prove existence and uniqueness of the solution to the BSDE. Moreover, we perform the optimization over a class of admissible contracts, instead of maximizing over the retention level. This feature allows us to cover a larger class of problems. Finally, we do not require the existence of an optimal control for the derivation of the BSDE, hence the general presentation turns out to be different.

The paper is organized as follows. In Section \ref{sec:model} we are going to introduce the risk model and to specify what information is available to the insurer. A rigorous mathematical construction is provided,  based on a measure change approach, necessary to develop the following analysis in full details.
In Section \ref{sec:filtering} the filtering problem is investigated in order to reduce the optimal stochastic control problem to a complete information setting. The stochastic differential equation satisfied by the filter is obtained, by exploiting both the Kushner-Stratonovich and the Zakai approaches.
In Section \ref{Optimal_Control} the optimal stochastic control problem is formulated, while in Section \ref{Value_Process} a characterization of the value process associated with the optimal stochastic control problem is illustrated. Due to the infinite dimension of the filter, the approach based on the Hamilton-Jacobi-Bellman equation cannot be exploited, so the value process is characterized as the unique solution of a BSDE. In Section \ref{Optimal_Reinsurance} the optimal reinsurance strategy is investigated  under general assumptions and some relevant cases are discussed.
Some proofs and useful computations are collected in Appendices \ref{Appendix}, \ref{app:useful_res} and \ref{AppendixC}.


\section{The mathematical model}\label{sec:model}

Let $(\Omega,\F, \P;\bF)$ be a filtered probability space and assume that the filtration $\bF=\{\F_t, \ t \in [0,T]\}$ satisfies the usual hypotheses. The time $T>0$ is a finite time horizon that represents the maturity of a reinsurance contract. Here we start by giving an overview of the optimal reinsurance problem from the primary insurer's point of view, then, in Section \ref{subsec:model}, we provide a rigorous construction of our model setting.

 Our aim is to introduce a dynamic contagion process which generalizes the Hawkes and Cox processes with shot noise intensity introduced e.g. by Dassios and Zhao \cite{Dassios_Zhao_2011}. More precisely, the claims counting process $N^\a$ has the following $(\P, \bF)$-stochastic intensity, for $t \in [0,T]$:
\begin{equation}\label{intensity}
\lambda_t = \beta + (\lambda_0 - \beta)  e^{-\alpha t} + \sum_{j=1}^{N^\a_t} e^{-\alpha (t -  T^\a_j)} \ell(Z^\a_j) +  \sum_{j=1}^{N^\b_t}  e^{-\alpha (t - T^\b_j)} Z^\b_j ,
\end{equation}
where
\begin{itemize}
\item  $\beta > 0$ is the constant reversion level;
\item  $\lambda_0 >0$ is the initial value;
\item $\alpha >0$ is the constant rate of exponential decay;
\item $N^\b$ is a Poisson process with constant intensity $\rho>0$;
\item $\{T^\a_n\}_{n \geq 1}$ are the jump times of $N^\a$, i.e., the time instants when claims are reported;
\item $\{T^\b_n\}_{n \geq 1}$  are the jump times of $N^\b$, i.e., when exogenous/external factors make intensity jump;
\item $\{Z^\a_n\}_{n \geq 1}$ represent the claim size and they are modeled as a sequence of i.i.d. $\R^+$-valued random variables with  distribution function $F^\a : (0, + \infty) \to [0,1]$ such that $\mathbb E[Z^\a] < + \infty$;
\item $\ell : [0, +\infty) \to [0, +\infty)$ is a measurable function (for instance we could take $\ell(z) = a z$, $a>0$, and the self-exciting jumps would be proportional  to  claims sizes) such that $\mathbb E[ \ell(Z^\a)] < + \infty$;
\item $\{Z^\b_n\}_{n \geq 1}$ are the externally-excited jumps and they are modeled as a sequence of i.i.d. $\R^+$-valued random variables with  distribution function $F^\b : (0, + \infty) \to [0,1]$, such that $\mathbb E[Z^\b] < + \infty$.
\end{itemize}

Notice that the counting process $N^{(1)}$ is defined via its intensity $\lambda$ in Equation \eqref{intensity}, which in turn depends on the history of $N^{(1)}$. So, an apparent logical loop seems to arise about the existence of $\lambda$. We postpone this issue to Section \ref{subsec:model}, where we perform a rigorous construction of the model based on an equivalent change of probability measure.

The following assumption will hold from now on:
\begin{assumption}\label{ass:indep}
We assume $N^{(2)}$,  $\{Z^\a_n\}_{n \geq 1}$ and $\{Z^\b_n\}_{n \geq 1}$ to be independent of each other.\\
\end{assumption}

We define the cumulative claim process $C = \{ C_t,  t \in [0,T] \}$ at time $t$ as
\begin{equation}\label{loss}
C_t = \sum_{j=1}^{N^\a_t} Z^\a_j,\quad t \in [0,T].
\end{equation}

\begin{remark}
Our model includes many meaningful properties of risk models. The claim arrival process has stochastic intensity, reflecting random changes in the instantaneous probability that accidents occur. Most importantly, our framework captures both self-exciting (endogenous) and externally-exciting (exogenous) factors, via, respectively, the claim arrival times and sizes $\{ T^\a_n, Z^\a_n \}_{n \geq 1}$ and $\{ T^\b_n, Z^\b_n \}_{n \geq 1}$. For this reason, it is well suited to describe, for instance, ca\-ta\-stro\-phic events, see Cao, Landriault and Li \cite{Cao_Landriault_Li}, where self-exciting jump sizes are independent on claims severity. Differently, in our model they depend on claim sizes: ${\color{red}{\ell}}(Z^{(1)}_j )$.
Moreover, the decay coefficient is considered, because the catastrophic events typically exhibit this behavior.
 \end{remark}

The insurance company is allowed to subscribe a reinsurance contract with a retention function $\Phi(z,u)$ parametrized by a dynamic reinsurance strategy $u_t\in U$, $\forall t\in[0,T]$ (the control). That is, under a dynamic strategy $u = \{ u_t, t \in [0,T]\}$ the aggregate losses covered by the insurer, denoted by $C^u=\{ C^u_t, t \in [0,T]\}$, read
$$
C^u_t = \sum_{j=1}^{N^\a_t} \Phi(Z^\a_j, u_{T^\a_j}), \quad t \in [0,T],
$$
so that the remaining losses $(C - C^u)$ will be undertaken by the reinsurer. We highlight that in our settings the insurer can choose the optimal reinsurance arrangement over a class of admissible contracts, see Section \ref{Optimal_Control} for details.
For this service a reinsurance premium rate $q^u = \{ q^u_t, t \in [0,T]\}$ must be paid.
Hence the primary insurer receives the insurance premium rate $c$, pays the reinsurance premium rate $q^u$ and bears the aggregate losses $C^u$, so that the surplus process, $R^u$, follows the SDE:
\[
dR^u_t = \bigl(c_t-q^u_t\bigl)\,dt - dC^u_t, \quad R_0^u = R_0 \in \mathbb R^+,
\]
where $R_0$ denotes the initial capital. Investing the surplus in a risk-free asset with interest rate $r>0$, the total wealth $X^u$ of the primary insurer is
\[
dX^u_t = dR^u_t + rX^u_t\,dt, \qquad X^u_0= R_0\in\mathbb{R}^+ .
\]

We assume that the information at disposal is limited: the insurer only observes the cumulative claims process $C$ in Equation \eqref{loss}. Let us denote by $\bH$ the natural filtration generated by $C$:
\begin{equation}\label{eq:defH}
\bH = \bF^C = \{\F_t^C, \ t \in [0,T]\} \subset \bF,  \qquad  \F_t^C =  \sigma {\{ C_s, 0 \le  s \le t \}}.
\end{equation}
We assume that the insurer and the reinsurer have the same information represented by $\bH$. Therefore, the insurance and the reinsurance premium have to be $\bH$-predictable. The same applies to the insurer's control $u$. The insurer aims at maximizing the expected exponential utility of terminal wealth over a suitable class of $\bH$-predictable strategies $\mathcal{U}$ (which will be made precise later in Definition \ref{def:U}):
\[\label{exp_utility}
\sup_{u\in\mathcal{U}}\mathbb{E}\bigl[ 1-e^{-\eta X^u_T} \bigr],
\]
 where $\eta>0$ denotes the insurer's risk aversion. More mathematical details on the control problem to be solved will be given in Section \ref{Optimal_Control}.

\begin{remark}
Notice that the stochastic wealth $X^u$ can possibly take negative values, due to the possibility of borrowing money from the bank account.
\end{remark}

This setting leads to investigate  a stochastic control problem under partial information. Due to  the presence of  the externally-excited component, the claim arrival intensity in Equation \eqref{intensity} is $\bF$-adapted rather than $\bH$-adapted, hence it is not observable by the insurance and reinsurance companies. We will reduce the original problem to a stochastic control problem under complete information by solving a filtering problem in Section \ref{sec:filtering}. The knowledge of the filter process allows to compute the $\bH$-adapted (predictable)  intensity of the claim arrival process $N^\a$, which represents the best estimate of the stochastic intensity $\lambda$ based on the available information. 


The next subsection provides a formal and rigorous construction of our model.

\subsection{Model construction}\label{subsec:model}

We are going to introduce the dynamic contagion model by a suitable measure change, starting from two Poisson processes with constant intensity on a given probability space  $(\Omega, \mathcal F, \Q; \mathbb F)$: $N^{(1)}$ is standard and $N^{(2)}$ has constant intensity $\rho>0$. Moreover, we take two sequences $\{Z^\a_n\}_{n \geq 1}$ and $\{Z^\b_n\}_{n \geq 1}$ of i.i.d. positive random variables with distribution functions $F^\a$ and $F^\b$, respectively, and such that  $\mathbb E^{\Q}[\ell(Z^\a)] < + \infty$ and $\mathbb \mathbb E^{\Q}[ Z^\b] <+ \infty$. We assume $N^{(1)}$, $N^{(2)}$,  $\{Z^\a_n\}_{n \geq 1}$ and $\{Z^\b_n\}_{n \geq 1}$ to be independent of each other under $\Q$.

The key idea behind our construction is to introduce a new measure $\P$, equivalent to $\Q$ on $(\Omega, \mathcal F; \mathbb F)$, such that, under $\P$, the intensity of $N^\b$ and the distributions of $\{Z^\a_n\}_{n \geq 1}$ and $\{Z^\b_n\}_{n \geq 1}$ do not change and $N^{(1)}$ is a counting  process with stochastic intensity $\lambda$  given by Equation \eqref{intensity}. Notice that, under $\P$, $N^{(1)}$, $N^{(2)}$,  $\{Z^\a_n\}_{n \geq 1}$ and $\{Z^\b_n\}_{n \geq 1}$ are not independent anymore.

Let us introduce the integer-valued random measures $m^\i(\ud t, \ud z)$, $i=1,2$
\begin{equation}
\label{eqn:m1m2}
m^\i(\ud t, \ud z) = \sum_{n \geq 1} \delta_{(T_n^\i, Z_n^\i)}(\ud t, \ud  z) 1\!\!1_{\{ T_n^\i <+ \infty\}},
\end{equation}
where $\delta_{(t,z)}$ denotes the Dirac measure in $(t,z)$.
Under  $\Q$, $m^\i(\ud t, \ud z)$, $i=1,2$, are independent Poisson measures  with compensator measures given respectively by
$$\nu^{\a, \Q}(\ud t, \ud z)=  F^\a(\ud z) \ud t, \quad \nu^{\b, \Q}(\ud t, \ud z)= \rho F^\b(\ud z) \ud t.$$

The measure change from $(\Q, \mathbb F)$ to $(\P,\mathbb F)$ will be performed via the stochastic process $L$ defined as follows, for $t \in [0,T]$:
\begin{equation}
\label{eqn:L1}
L_t = \E \left( \int_0^t \int_0^{+\infty} (\lambda_{s ^-}- 1) \left( m^\a(\ud s, \ud z)  - F^\a(\ud z) \ud s \right) \right) = \E \left( \int_0^t(\lambda_{s ^-}- 1) (dN_s^\a  - \ud s) \right),
\end{equation}
where $\E (M_t)$ denotes the Dol\'eans-Dade exponential of a martingale $M$ and where $\lambda$ under $\mathbb Q$ is defined by Equation \eqref{intensity}.
This process will be proved to be a $(\Q, \mathbb F)$-martingale under the following:

\begin{assumption}\label{nuova_bis}
We assume that there exists $\varepsilon >0$ such that
$$\mathbb E^{\Q}\left[ e^{\varepsilon \ell(Z^\a)} \right] <+ \infty,
	\quad \mathbb E^{\Q} \left[ e^{\varepsilon Z^\b} \right] <+ \infty.$$
\end{assumption}
Before proving the martingale property, we notice the following:
\begin{remark}
Let us observe that
$\{ \int_0^t(\lambda_{s ^-}- 1) (dN_s^\a  - \ud s)$, $t \in [0,T]\}$  is a $(\Q, \mathbb F)$-martingale, since $\mathbb E^{\Q} [\int_0^t\lambda_{s } \ud s ]<+ \infty$, $\forall t \in [0,T]$.
In fact, by Equation \eqref{intensity}
\begin{equation}\label{eq:lambda_bound}
\lambda_t \le \max\{\lambda_ 0,\beta\} +  \sum_{j=1}^{N^\a_t}  \ell(Z^\a_j) +  \sum_{j=1}^{N^\b_t}  Z^\b_j
\end{equation}
and
\[
\begin{split}
\mathbb E^{\Q} [\lambda_{t } ]  &\le \max\{\lambda_ 0,\beta\}  + E^{\Q} [N^\a_t] E^{\Q} [\ell(Z^\a)] 		+ E^{\Q} [N^\b_t] E^{\Q} [Z^\b] \\
	&=\max\{\lambda_ 0,\beta\}   + (E^{\Q} [\ell(Z^\a] + \rho E^{\Q} [Z^\b] ) t .
\end{split}
\]
\end{remark}

We then have an explicit expression for $ L_t $:
\begin{equation}
\label{eqn:L}
L_t = e^{ -\int_0^t (\lambda_{s} - 1) ds  + \int_0^t \ln (\lambda_{s^-})  d N_s^{(1)} }, \quad t \in [0,T] ,
\end{equation}
and we define the equivalent measure $\P$ via
$$
\frac{d \P}{d \Q} {\Big \vert  _{\mathcal F_T}} = L_T.
$$

\begin{proposition}\label{Prop:L_mart}
Under Assumption \ref{nuova_bis} the Radon-Nikodym density process $L$ given in Equation \eqref{eqn:L} is a $(\Q, \mathbb F)$-martingale.
\end{proposition}
\begin{proof}
This proof is based on Sokol and Hansen \cite[Corollary 2.5]{sokol}. We observe that $\lambda$ in Equation \eqref{intensity} is nonnegative, predictable and locally bounded. Hence Sokol and Hansen \cite[Corollary 2.5]{sokol} can be straightforwardly applied after we prove that condition $(2.6)$ therein holds: there exists $\varepsilon >0$ such that whenever $0 \le u \le t, t-u \le \varepsilon$
\begin{equation}
\label{eqn:sokol_cond}
\mathbb E^{\Q} \left[ e^{ \int_u^t \log_+ (\lambda_{s-} ) d N_s^{(1)} } \right] < +\infty,
\end{equation}
where $\log_+ (x):=\max\{ 0, \log x\}$.
Applying Lemma \ref{lemma:exp_predictable} under the measure $\Q$, we obtain that
$$
\mathbb E^{\Q} \left[ e^{\int_u^t \log_+ (\lambda_{s-}) \,dN_s^{(1)}  }\right]
	\le  \mathbb E^{\Q}  \left[ e^{  \int_u^t( \lambda_{s-} -1) \,ds}\right].
$$
Hence condition \eqref{eqn:sokol_cond} is fulfilled if the expectation $\mathbb E^{\Q}  \left[ e^{  \int_u^t \lambda_{s-} \,ds}\right]$ is finite. By Equation \eqref{eq:lambda_bound} we have:
$$
\mathbb E^{\Q}  \left[ e^{  \int_u^t \lambda_{s-}  \,ds}\right] \le e^{\varepsilon \lambda_0\lor\beta} \mathbb E^{\Q}  \left[ e^{  \varepsilon  \sum_{j=1}^{N^\a_t}  \ell(Z^\a_j)  } \cdot e^{  \varepsilon  \sum_{j=1}^{N^\b_t}  Z^\b_j    } \right] \le e^{\varepsilon \lambda_0\lor\beta} \mathbb E^{\Q}  \left[ e^{  \varepsilon  \sum_{j=1}^{N^\a_t}  \ell(Z^\a_j)  } \right]  \mathbb E^{\Q}  \left[ e^{  \varepsilon  \sum_{j=1}^{N^\b_t}  Z^\b_j    } \right],
$$
where we used the mutual independence of $N^{(1)}$, $N^{(2)}$,  $\{Z^\a_n\}_{n \geq 1}$ and $\{Z^\b_n\}_{n \geq 1}$ which holds by construction under $\Q$. By exploiting Lemma \ref{lemma:exp_predictable2} we immediately find:
\begin{eqnarray*}
	 \mathbb E^{\Q}  \left[ e^{  \varepsilon  \sum_{j=1}^{N^\a_t}  \ell(Z^\a_j)  } \right]
	& = & 
e^{ t  \big (\mathbb E^{\Q}   [ e^{  \varepsilon \ell(Z^\a)  }  ] - 1 \big ) } < +\infty.
\end{eqnarray*}
Using similar arguments, one shows that
$
\mathbb E^{\Q}  \left[ e^{  \varepsilon  \sum_{j=1}^{N^\b_t}  Z^\b_j    } \right]  = e^{\rho t \big (\mathbb E^{\Q}   [ e^{  \varepsilon Z^\b  }   ] - 1 \big ) } < +\infty.
$
\end{proof}

Now that the change of measure has been rigorously introduced, we can safely introduce the $(\P, \bF)$ compensator measures of $m^\i(\ud t, \ud z), i=1,2$.
\begin{remark}\label{projection}
By the Girsanov Theorem the $(\P, \bF)$-predictable projections measures (the so-called compensator measures) of $m^\a(dt,dz)$ and $m^\b(dt,dz)$ (see Equation \eqref{eqn:m1m2}) are given respectively by
\begin{equation} \label{dualpred}
\nu^\a(\ud t, \ud z)= \lambda_{t^-} F^\a(\ud z) \ud t, \quad \nu^\b(\ud t, \ud z)= \rho F^\b(\ud z) \ud t.
\end{equation}
In particular, $N^\a$ is a point process with $(\P, \bF)$-predictable intensity $\{\lambda_{s^-}\}_{s \in [0,T]}$, while $N^\b$ remains a point process with constant $(\P, \bF)$ intensity $\rho>0$.\\
It turns out that for any $\bF$-predictable random field $\{H(t,z), t \in [0,T], z \in [0, + \infty)\}$ and $i=1,2$
$$
\bE \left[ \int_0^t \int_0^{+\infty} H(s,z)m^\i(\ud s, \ud z) \right] = \bE \left[ \int_0^t \int_0^{+\infty} H(s,z)\nu^\i(\ud s, \ud z) \right], \ {\color{red}{\forall}} t \in [0,T],
$$
where $\nu^\i(\ud s, \ud z)$, $i=1,2$, are defined in Equation \eqref{dualpred}.
Moreover, under the condition
$$\bE \left[ \int_0^T \int_0^{+\infty} |H(s,z)|\nu^\i(\ud s, \ud z) \right] <+ \infty,$$
the process
$$ \int_0^t \int_0^{+\infty} H(s,z) \left( m^\i(\ud s, \ud z) - \nu^\i(\ud s, \ud z) \right), \quad t \in [0,T],$$
is a $(\P, \bF)$-martingale.
\end{remark}

\subsection{Markov property}
\label{Markov property}

In this subsection we discuss and characterize the Markov structure of the intensity, working on $(\Omega,\F, \P;\bF)$. Equation \eqref{intensity} reads as
\begin{equation}\label{intensity_eq}
\ud \lambda_t = \alpha (\beta - \lambda_t ) \ud t+ \int_0^{+ \infty} \ell(z) m^\a(\ud t , \ud z) + \int_0^{+ \infty}  z m^\b(\ud t , \ud z).
\end{equation}


\begin{proposition}\label{prop:generatore}
The process $\lambda$ is a $(\P,\bF)$-Markov process with generator
\begin{equation}\label{generatore}
{\L} f( \lambda) =\alpha (\beta - \lambda) f' (\lambda) +
\int_0^{+ \infty} [ f(\lambda + \ell(z)) - f(\lambda) ] \lambda F^\a (\ud z) \nonumber + \int_0^{+ \infty} [ f(\lambda + z) - f(\lambda) ]  \rho F^\b (\ud z).
\end{equation}

The domain of the generator ${\L}$, denoted by ${\mathcal D}(\L)$, is given by the class of functions $f\in C^{1}(0, + \infty)$ such that
\begin{equation}
 \begin{split}
 &\bE \left[ \int_0^t \int_0^{+ \infty} |f(\lambda_s +\ell(z)) - f(\lambda_s)| \lambda_s  F^\a (\ud z)\ud s \right] <+ \infty, \\
 &\bE \left[ \int_0^t \int_0^{+ \infty}  |f(\lambda_s + z)) - f(\lambda_s)|  F^\b (\ud z)\ud s \right] <+ \infty,
 \end{split}
 \end{equation}
and
  \begin{equation}
   \bE \left[ \int_0^t \lambda_s |  f' (\lambda_s) | \ud s \right] <+ \infty.
  \end{equation}
\end{proposition}

\begin{proof}
It is a direct  application of It\^o formula.
\end{proof}
In what follows we will need the following, which will be crucial to prove Proposition \ref{mom}:
\begin{assumption}\label{nuova1}
\[
 \bE[ (\ell(Z^\a)^k] < + \infty, \quad  \bE[ (Z^\b)^k] < + \infty, \quad \forall k=1,2, \dots.
\]
\end{assumption}

\begin{proposition}\label{mom}
Under Assumption \ref{nuova1}, for any $t \in [0,T]$
$$\bE \left[ \int_0^t \lambda_s^k \ud s \right] <  + \infty, \quad \forall k=1,2, \dots.$$
\end{proposition}

\begin{proof}
We proceed by induction on $k$. We first prove that $\bE[ \lambda_t ] \leq h_1(t)$,  $t\geq 0$, with $h_1$ a measurable, nonnegative function such that $\int_0^T h_1(t) \ud t < + \infty$.  Let us observe that Equation \eqref{intensity} reads as
\begin{equation}
\lambda_t = \beta + (\lambda_0 - \beta)  e^{-\alpha t} +\int_0^t \int_0^ {+ \infty} e^{-\alpha (t -  s)} \ell(z) m^\a(\ud s , \ud z) +  \int_0^t \int_0^ {+ \infty} e^{-\alpha (t -  s)} z m^\b(\ud s , \ud z),
\end{equation}
hence by Remark \ref{projection}
\begin{equation}
\bE \left[ \lambda_t \right]= \beta + \left( \lambda_0 - \beta - \frac{\rho \bE[Z^\b] }{\alpha} \right)  e^{-\alpha t} + {1\over \alpha} \rho \bE[Z^\b]  +  \bE[\ell(Z^\a)] \int_0^t  e^{-\alpha (t -  s)}  \bE[\lambda_s]  \ud s.
\end{equation}
By applying  Gronwall's Lemma we obtain
\begin{equation} \label{stim}
\bE[\lambda_t ] \leq  \Big ( \beta + (\lambda_0 - \beta - \frac{\rho \bE[Z^\b] }{\alpha} )  e^{-\alpha t} + {1\over \alpha} \rho \bE[Z^\b] \Big ) e^{\bE[\ell(Z^\a)] {1 - e^{-\alpha t}\over \alpha}} = h_1(t), \quad t \in [0,T].
\end{equation}

It is immediate to verify that $ h_1(t) \geq 0$ and $\int_0^T h_1(t) \ud t < + \infty$.
Let us assume that $\bE [ \lambda_t^i ] \leq h_i(t)$,  with $h_i$ a measurable, nonnegative function such that $\int_0^T h_i(t) \ud t < + \infty$ for any $i=1,2, \dots k-1$.  By It\^o formula we get
\[
\begin{split}
\lambda_t^k   &  =  \lambda_0^k +  \int_0^t \alpha (\beta -\lambda_s) k \lambda^{k-1} _s \ud s +  \int_0^t \int_0^ {+ \infty} [  (\lambda_{s^-} + \ell(z))^k - (\lambda_{s^-})^k ] m^\a(\ud s , \ud z) \\
 \;&+\int_0^t  \int_0^{+ \infty} [ (\lambda_{s^-} + z)^k - (\lambda_{s^-})^k ] m^\b(\ud s , \ud z)  \\
& = \lambda_0^k + \int_0^t \alpha (\beta -\lambda_s) k \lambda^{k-1} _s \ud s + \int_0^t \int_0^ {+ \infty} \sum_{i=0}^{k-1} \binom{k}{i} (\lambda_{s^-})^i \ell(z)^{k-i} m^\a(\ud s , \ud z)   \\
 \;&+ \int_0^t  \int_0^{+ \infty} \sum_{i=0}^{k-1}  \binom{k}{i} (\lambda_{s^-})^i z^{k-i} m^\b(\ud s , \ud z).\end{split}
\]

Then, there exist $c_i >0$,   $i=1,2, \dots k$ such that
\begin{eqnarray*}
\bE[ \lambda_t^k]   &  = & \lambda_0^k +  \int_0^t \alpha k (\beta  \bE[\lambda^{k-1} _s] -  \bE[\lambda_s^k])  \ud s + \int_0^t \sum_{i=0}^{k-1} \binom{k}{i}  \bE[ \lambda_s^{i+1}]   \bE[\ell(Z^\a)^{k-i}] \ud s  \\
 & + &  \int_0^t  \sum_{i=0}^{k-1} \binom{k}{i} \bE[\lambda^i_{s}] \rho \bE[(Z^\b)^{k-i}] \ud s  \leq \lambda_0^k +  \int_0^t  \sum_{i=0}^{k-1}  c_i  h_i(s) \ud s + \int_0^t c_k \bE[ \lambda_s^k]  \ud s,
 \end{eqnarray*}
and again by  Gronwall's Lemma it follows that $\bE[ \lambda_t^{k} ] \leq h_{k}(t)$,  with $h_{k}$ a measurable, integrable and  nonnegative function on $[0,T]$, and this concludes the proof.
\end{proof}

\begin{proposition}
Under Assumption \ref{nuova1}, the functions $f_k(\lambda) :=\lambda^k $, $k=1,2, \dots$ belong to $\mathcal{D}(\L)$.
\end{proposition}
\begin{proof}
Under Assumption \ref{nuova1}, by computations similar to those performed in the proof of Proposition \ref{mom}, we get the claim.
\end{proof}


\section{The filtering problem}\label{sec:filtering}

We assume that the  insurance company has a partial information because the externally-exciting component in the intensity process $\lambda$ introduced in Equation \eqref{intensity} is not observable. For filtering of Cox processes with shot noise intensity, that is without the self-exciting component in Equation \eqref{intensity}, we refer to Dassios and Jang \cite{Dassios_Jang}, where the estimation of the intensity $\lambda$ given the observations of the claim arrival process $N^\a$ reduces to the use of the classical Kalman-Bucy filter after a Gaussian approximation of the intensity is performed. This result applies in the case where the intensity $\rho$ of the externally-exciting component  is sufficiently large.
Their working setting can be seen as a particular case of our contagion model and their results can then be obtained as special cases, with no assumption on $\rho$ needed (see also Remark \ref{Cshot}).

The insurance company aims at estimating the intensity $\lambda$ by observing the cumulative claim process $C$ defined in Equation \eqref{loss}, that is, by observing the double sequence $\{(T_n^\a, Z_n^\a)\}_{n \geq 1}$ of arrival times and claim sizes.  This leads to a filtering problem with marked point processes observations.

Let us recall that $\bH = \bF^C$, defined in Equation \eqref{eq:defH}, is the observation flow, representing the  information at disposal by the insurance company.
So, the estimate of the intensity $\lambda$ can be described through  the filter process $\pi=\{\pi_t, t \in [0,T]\}$ which provides the conditional distribution of $\lambda_t$ given $\H_t$, for any time $t \in [0,T]$.
More in details, the filter is the $\bH$-c\`adl\`ag (right-continuous with left limits) process taking values in the space of probability measures on $[0, + \infty)$ such that
 $$
 \pi_t(f) = \bE[ f(\lambda_t) |  \H_t ],
 $$
for any function $f$ satisfying $\bE[ \int_0^t | f(\lambda_s)| \ud s] < +\infty$, $\forall t \in [0,T]$.
It is easy to verify that $\{\pi_{t^-}(\lambda), t \in [0,T]\}$, where $\pi_{t}(\lambda) =  \bE[ \lambda_t |  \H_t ]$ and $\pi_{t^-}(\lambda) = \lim_{s\to t^-} \pi_{t}(\lambda)$, provides the $\bH$-predictable intensity of $N^\a$.

\begin{remark}\label{nuovo}
For any function $f$ satisfying $\bE[ \int_0^t | f(\lambda_s)| \ud s] < +\infty$, for any $ t\in [0,T]$, we have that $\bE \left[ \int_0^t  \pi_s(f) \ud s \right]  = \bE \left[ \int_0^t f(\lambda_s) \ud s \right]$  and Jensen's inequality implies
$$\bE  \left[  \int_0^t |\pi_s(f)| \ud s\right]  \leq \bE  \left[  \int_0^t \pi_s(|f|)\ud s\right] = \bE  \left[  \int_0^t| f(\lambda_s)| \ud s\right]  < +\infty, \quad \forall t\in [0,T].$$
\end{remark}

By applying the innovation method (see for instance Br\'emaud \cite[Chapter IV]{Bremaud}) we will characterize the filter in terms of the so called Kushner-Stratonovich (KS henceforth) equation.

\begin{theorem}[Kushner-Stratonovich equation]  \label{KSthm}
For any $f \in {\mathcal D}(\L)$, the filter  is the unique strong solution to the filtering equation, for any $t \in [0,T]$
\begin{equation}\label{KS}
\begin{split}
 \pi_t (f) &= f(\lambda_0) + \int_0^t \pi_s (\L f) \ud s  \\
&+ \int_0^t  \int_0^{+ \infty} \Big ( {\pi_{s^-}(f( \lambda + \ell(z)) \lambda)  \over
\pi_{s^-}(\lambda)} - \pi_{s^-}(f) \Big ) \left( m^\a( \ud s, \ud z) - \pi_{s^-} (\lambda) F^\a (\ud z) \ud s \right),
\end{split}
\end{equation}
where $\L$ and ${\mathcal D}(\L)$ are given in Proposition \ref{prop:generatore}.

\end{theorem}

 \begin{proof}
 We denote by $\widehat R$ the $(\P, \bH)$-optional projection of  an $\bF$-progressively measurable process $R$ such that  $\bE[ |R_t|] <+ \infty$ $\forall t \in[0,T]$. We will use the two-well known facts:
 \begin{itemize}
 \item for every $(\P, \bF)$-martingale $m$, the $(\P, \bH)$-optional projection $\widehat m$  is a $(\P, \bH)$-martingale;
\item for any $\bF$-progressively measurable process $\Psi$ we have that $ \widehat {\int_0^t \Psi_s \ud s} - {\int_0^t \widehat \Psi_s \ud s}$, $\in[0,T]$,
 is a $(\P,\bH)$-martingale.
 \end{itemize}

By It\^o formula, for any $f \in {\mathcal D}(\L)$, we have:
$$
f(\lambda_t) = f(\lambda_0) + \int_0^t {\L} f(\lambda_s) \ud s  +  m^f_t, \quad \forall t \in [0,T],
$$
where $m^f$ is a $(\P, \bF)$-martingale and taking the $(\P,\bH)$-optional projection we get
\begin{equation} \label{dec2}
\widehat {f(\lambda_t)} = f(\lambda_0) + \int_0^t \widehat {\L f(\lambda_s)} \ud s  +  M^f_t, \quad \forall t \in [0,T],
\end{equation}
where $M^f$ is a $(\P, \bH)$-martingale.  By the martingale representation theorem there exists an $\bH$-predictable random field,
$h^f=\{ h^f_t(z), t \in [0,T], z \in [0, + \infty)\}$, such that, for any $t \in [0,T]$
\begin{equation}\label{mgf}
M^f_t =  \int_0^t \int_0^ {+ \infty} h^f_s(z)  \left( m^\a(\ud s , \ud z)  - \pi_{s^-}(\lambda) F^\a (\ud z)\ud s \right)
\end{equation}
and $\bE \left[ \int_0^t \int_0^ {+ \infty} |h^f_s(z)| \pi_{s^-}(\lambda) F^\a (\ud z)\ud s \right] <+ \infty$. To derive  the expression of $h^f$, we consider an $\bH$-adapted and bounded process
$$\Gamma_t = \int_0^t  \int_0^ {+ \infty} U_s(z) m^\a(\ud s, \ud z)$$
with $U$ an $\bH$-predictable bounded random field. Since $\Gamma$ is $\bH$-adapted the following equality holds
\begin{equation}\label{1} \widehat{\Gamma_t f(\lambda_t)} = \Gamma_t \widehat{f(\lambda_t)},  \quad \forall t \in [0,T], \P-a.s. .
\end{equation}
By applying the product rule we get
\[
\begin{split}
\ud (\Gamma_t f(\lambda_t) ) & =  \Gamma_{t^-}  \ud f(\lambda_t) + f(\lambda_{t^-}) \ud \Gamma_t + \ud ([ \Gamma_t, f(\lambda_t) ]) \\
& = \Gamma_{t^-}  {\L}f(\lambda_t) \ud t  + \Gamma_{t^-} \ud m^f_t + \int_0^ {+ \infty} f(\lambda_{t^-}) U_t(z) m^\a(\ud t , \ud z)   \\
 \;& +  \int_0^{+ \infty} U_t (z) [ f(\lambda_{t^-} + \ell(z)) - f(\lambda_{t^-})] m^\a( \ud t, \ud z)  \\
& = \Gamma_{t^-}  {\L}f(\lambda_t) \ud t  +
\int_0^{+ \infty} U_t (z)  f(\lambda_{t^-} + \ell(z)) \lambda_{t} F^\a(\ud z) \ud t+ \ud \overline  m^f_t ,
\end{split}
\]
 where   $\overline m^f$ is a $(\P,\bF)$-martingale. Taking the  $(\P, \bH)$-optional projection we obtain that
\begin{align}\label{eq1}
	\ud ( \widehat{\Gamma_t f(\lambda_t)} ) = & \Big [\Gamma_{t^-} \  \widehat{{\L}f(\lambda_t)} +  \int_0^{+ \infty} U_t(z)
	\widehat{\big (\lambda_t f(\lambda_{t} + \ell}(z)) \big )  F^\a (\ud z) \Big ] \ud t + \M^f_t,
\end{align}
where $\M^f$ is a $(\P, \bH)$-martingale.  On the other hand we have that
\begin{equation} \label{eq2}
\begin{split}
\ud (\Gamma_t \ \widehat{f(\lambda_t)} ) &= \Gamma_{t^-}  \ud \widehat{f(\lambda_t)}  + \widehat{f(\lambda_{t^-})} \ud \Gamma_t + \ud ([ \Gamma_t, \widehat{f(\lambda_t)}  ])  \\
 &=  \Big [ \Gamma_{t^-}  \widehat{{\L}f(\lambda_t)}  +   \int_0^{+ \infty} U_t(z) \big ( h^f_t (z) +   \widehat{f(\lambda_{t})} \big )\widehat{\lambda_t}  F^\a (\ud z) \Big ] \ud t + \overline \M^f_t,
\end{split}
\end{equation}
where $\overline \M^f$ is a $(\P, \bH)$-martingale. By \eqref{1} we have that the finite variation parts in  Equations \eqref{eq1} and \eqref{eq2} have to coincide: for any $t \in [0,T]$
\begin{align}
\int_0^t \int_0^{+ \infty} U_s (z) h^f_s (z) \widehat{\lambda_s}  F^\a (\ud z) = \int_0^t \int_0^{+ \infty} U_s(z) \big [  \widehat{(  \lambda_s f(\lambda_{s} + \ell}(z)))-   \widehat{f(\lambda_{s})} \widehat{\lambda_s} ] F^\a (\ud z).
\end{align}
We select $U_t (z)$ of the form $U_t(z) = U_t 1\!\!1_A(z) 1\!\!1_{ \{t \leq T^\a_n \}}$ with $U = \{U_t, t  \in [0,T]\}$ any bounded $\bH$-predictable, positive process and $A \in \mathcal B([0, + \infty))$.  With this choice we get that $\Gamma$ is bounded and $\forall A \in \mathcal B([0, + \infty))$ and $t \leq T^\a_n \wedge T$
$$\int_A  h^f_t (z) \widehat{\lambda_t}  F^\a (\ud z) = \int_A \big [  \widehat{( \lambda_t f(\lambda_{t} + \ell}(z))) -   \widehat{f(\lambda_{t})} \widehat{\lambda_t} ] F^\a (\ud z)$$
and recalling that $\lambda_{t} >0$, $\forall t \in [0,T]$ (which implies  $\widehat{\lambda_t} = \pi_t(\lambda)>0$ $\forall t \in [0,T]$), we obtain that
\begin{equation}
h^f_t (z)  = {\pi_{t^-}(f(\lambda + \ell(z)) \lambda)  \over \pi_{t^-}(\lambda)} - \pi_{t^-}(f(\lambda)), \quad    t \leq T^\a_n \wedge T.
\end{equation}
Finally since the counting process $N^\a$ is not explosive we have that $T^\a_n \to + \infty$ as $n \to + \infty$ and by Equations \eqref{dec2} and \eqref{mgf} we obtain that the filter is solution to the KS Equation \eqref{KS}.

It remains to prove uniqueness for this equation. As in Theorem 3.3 in Ceci and Colaneri \cite{Ceci_Colaneri} we have that strong uniqueness of the solution to the KS Equation  follows by uniqueness of the Filtered Martingale Problem (FMP($\bar \L, \lambda_0, C_0$)) associated to the generator $\bar \L$ of the pair $\{ (\lambda_t, C_t), \in [0,T] \}$ for any initial condition $(\lambda_0, C_0) \in (0,+\infty) \times [0, + \infty)$. For details on FMP we refer to Kurtz and Ocone \cite{Kurtz_Ocone}.
The operator $\bar \L$ is given by
\begin{eqnarray}\label{generatore1}
{\bar \L} f( \lambda, C) &=&\alpha (\beta - \lambda){\partial f \over \partial \lambda} (\lambda,C) +
\int_0^{+ \infty} [ f(\lambda + \ell(z), C+z) - f(\lambda,C) ] \lambda F^\a (\ud z) \nonumber \\
& & + \int_0^{+ \infty} [ f(\lambda + z, C) - f(\lambda, C) ]  \rho F^\b (\ud z),
\end{eqnarray}
for a suitable class of functions $f(\lambda, C)$.

Next, to prove that  the FMP($\bar \L, \lambda_0, C_0$) has a unique solution we apply Theorem 3.3 in Kurtz and Ocone \cite{Kurtz_Ocone}, after checking that the  required hypotheses are fulfilled.  First, let us observe that the martingale problem for the operator $\bar \L$ is well posed on the space of c\`adl\`ag $(0, + \infty) \times [0, + \infty)$-valued paths. Furthermore, we can choose a domain ${\mathcal D}(\bar \L)$, such that for any $f \in {\mathcal D}(\bar \L)$ then $\L f \in C_b((0, + \infty) \times [0, + \infty))$. Let ${\mathcal D}(\bar \L)$ be the set of functions $f \in C((0, + \infty) \times [0, +\infty))$ having compact support and $C^1$ w.r.t. $\lambda \in (0, +\infty)$. Then for any $f \in {\mathcal D}(\bar \L)$ there exists $R_f>0$ such that
$$|\bar \L f(\lambda, C)| \leq \Big  \{ \alpha (\beta + \lambda) \Big |{\partial f \over \partial \lambda} (\lambda, C) \Big |+ 2 |\!| f |\!|(\lambda + \rho)\Big \}\mathbbm{1}_{\{ |\!| (\lambda, C) |\!| \leq R_f\}} \leq K$$
with $K$ positive constant. Moreover,  it is easy to verify that $\bar \L f(\lambda, C)$ is a continuous function of their arguments.
Finally, ${\mathcal D}(\bar \L)$ is dense in the space of continuous functions which vanish at infinity and so all  hypotheses of Theorem 3.3 in Kurtz and Ocone \cite{Kurtz_Ocone}  are satisfied and this concludes the proof.
\end{proof}
The filtering Equation \eqref{KS} has a natural recursive structure in terms of the sequence $\{T^\a_n\}_{n \geq 1}$. Indeed, between two consecutive jump times, for $t \in [T^\a_n \wedge T, T^\a_{n+1} \wedge T)$ Equation \eqref{KS} reads as
\begin{equation}\label{ks1}
\ud \pi_t (f)= \pi_t (\widetilde \L f) \ud t - [ \pi_{t} (\lambda f) -  \pi_{t} (\lambda) \pi_{t}(f) ] \ud t,
\end{equation}
where
\begin{equation}\label{L tilde} \widetilde \L f(\lambda) = \alpha (\beta - \lambda )f'( \lambda) +
\int_0^{+ \infty} [ f(\lambda + z) - f(\lambda) ]  \rho F^\b (\ud z).\end{equation}

At a jump time $T^\a_n \leq T$, we have that the value of the filter is completely determined by the knowledge
of the filter $\pi_t$, with $t \in (T^\a_{n-1} \wedge T, T^\a_n\wedge T)$ and  the observed data $(T^\a_n, Z^\a_n)$, precisely
\begin{equation}\label{ks1jump}\pi_{T^\a_n}(f) = {\pi_{T^\a_{n^-}}( \lambda f(\lambda + \ell(Z^\a_n))) \over \pi_{T^\a_{n^-}}( \lambda)}.
\end{equation}
Notice that $\widetilde \L $ is the Markov generator of a shot noise Cox process, obtained taking $\ell(z) =0$ in Equation \eqref{intensity}.
\begin{remark}
Let us consider $f_k (\lambda) = \lambda^k$, $k=1,2, \dots$  since
$$\widetilde {\L} f_k (\lambda) = \alpha (\beta - \lambda) k f_{k-1} (\lambda) +
\int_0^{+ \infty} [(\lambda + z)^k   - \lambda^k] \rho F^\b (\ud z)$$
we get by Equations \eqref{ks1} and  \eqref{ks1jump}, that, for  any $k=1,2, \dots$,
between two consecutive jump times
\begin{align}\label{ks2}
\ud \pi_t (f_k)& = \alpha \big (\beta  \pi_t (f_{k-1}) - \pi_t (f_{k}) \big ) k  \ud t  \notag\\
&+ \sum_{i=0}^{k-1} \binom{k}{i} \pi_t (f_i)
 \rho \bE[ (Z^\b)^{k-i} ] \ud t -
  [ \pi_{t} (f_{k+1}) -  \pi_{t} (f_1) \pi_{t}(f_k)) ] \ud t
  \end{align}
  and at a jump time $T^\a_n \leq T$
  \begin{equation} \label{ks2jump}
  \pi_{T^\a_n}(f_k) = {\pi_{T^\a_{n^-}}( \lambda (\lambda + \ell(Z^\a_n))^k) \over \pi_{T^\a_{n^-}}( f_1)}  = { \sum_{i=0}^{k} \binom{k}{i}  \pi_{T^\a_{n^-}}( f_{i+1})  \ell(Z^\a_n)^{k-i}  \over \pi_{T^\a_{n^-}}( f_1)}.
\end{equation}

In particular, for $k=1$ we have that $\pi_{t} (f_1) = \pi_{t}(\lambda)$ provides the $(\P, \bH)$-intensity of $N^\a$, and the KS equation reads as
\begin{eqnarray}
 \pi_t (\lambda)  & = & \lambda_0 + \int_0^t \pi_s (\L f_1) \ud s  \nonumber \\
&&+ \int_0^t  \int_0^{+ \infty} \Big ( {\pi_{s^-}( (\lambda + \ell(z)) \lambda)  \over \pi_{s^-}(\lambda)} - \pi_{s^-}(\lambda) \Big )(m^\a( \ud s, \ud z) - \pi_{s^-} (\lambda) F^\a (\ud z) \ud s) \nonumber \\
& = & \lambda_0 + \int_0^t \Big [\alpha(\beta - \pi_s (\lambda) )+ \rho \bE[Z^\b] - (\pi_s(\lambda^2) - \pi_s(\lambda)^2) \Big ] \ud s  \nonumber \\
&&+ \int_0^t  \int_0^{+ \infty} \Big [ \ell(z) + {\pi_{s^-}(\lambda^2) - \pi_{s^-}(\lambda)^2  \over \pi_{s^-}(\lambda) } \Big ]m^\a( \ud s, \ud z),\nonumber
\end{eqnarray}
that is
\begin{eqnarray}\label{KS_lambda}
 \ud \pi_t (\lambda) &= &\alpha \Big (\beta +\rho  {\bE[Z^\b] \over \alpha} - \pi_t (\lambda) \Big) \ud t - (\pi_t(\lambda^2) - \pi_t(\lambda)^2)  \ud t  \nonumber \\
 & & + \int_0^{+ \infty} \ell(z) m^\a( \ud s, \ud z) + {\pi_{t^-}(\lambda^2) - \pi_{t^-}(\lambda)^2  \over \pi_{t^-}(\lambda) } \ud N^\a_t.
 \end{eqnarray}

Notice that the equations for  $\pi_t (f_k)$ depend on $\pi_t (f_1), \dots, \pi_{t} (f_{k+1})$, for any $k=1,2, \dots$. Thus the $(\P, \bH)$-predictable intensity of $N^\a$, $\pi_{t^-}(\lambda)= \pi_{t^-}(f_1)$, is completely characterized  by  a countable systems of equations given in \eqref{ks2} and \eqref{ks2jump}. Moreover, Equation \eqref{KS_lambda} involves the processes $\pi_t (\lambda)$ and $\pi_{t}(\lambda^2) - \pi_{t}(\lambda)^2= \bE[ (\lambda_t -  \pi_{t}(\lambda))^2  | \H_t) = \text{Var} (\lambda_t | \H_t)$.
 \end{remark}

\begin{remark}\label{ultimo}
By Jensen's inequality, since $\pi_t(\lambda^2)  \geq \pi_t(\lambda)^2$,
we get by Equation \eqref{KS_lambda} and a comparison result that
$$
\pi_t (\lambda)  \leq Y_t, \quad \P-a.s. \ \forall t \in [0,T],
$$
where the process $Y$ has the same jumps of  $\pi(\lambda)$ and between two consecutive jumps solves the SDE: $\ud Y_t = \alpha(\widetilde \beta - Y_t)  \ud t,$
where $\widetilde \beta = \beta + {\rho \bE[Z^\b] \over \alpha}.$ More precisely, for  $t \in [T^\a_{n}\wedge T, T^\a_{n+1}\wedge T), Y_t =   \widetilde \beta  + (\pi_{T^\a_n}(\lambda)  -   \widetilde \beta) e^{-\alpha(t - T^\a_{n})}.$
Hence the filter is dominated by a process with exponential decay behaviour between consecutive jump times.
\end{remark}

Thanks to Theorem \ref{KSthm} we have characterized the filter  in terms of a nonlinear stochastic equation. In our framework it is possible to describe  the filter also in terms of the unnormalized filter as solution of the so-called Zakai equation, which has the advantage of being linear. 

By the Kallianpur-Striebel formula we get that, for any $t \in [0,T]$
\begin{equation}
\pi_t(f) =  {\bE^\Q[ L_t f(\lambda_t) |  \H_t ]  \over  \bE^\Q[ L_t  |  \H_t ]}  = {\sigma_t (f) \over  \sigma_t (1)},
\end{equation}
where $\Q$ is the equivalent probability measure introduced in Section \ref{subsec:model}, $L$ is given in Equation \eqref{eqn:L}.
The process ${\sigma_t (f)} = {\bE^\Q[ L_t f(\lambda_t) |  \H_t ]}$, $t \in [0,T]$,  denotes the unnormalized filter and is a finite measure-valued $\bH$-c\`adl\`ag process.

\begin{proposition}[Zakai equation]\label{Zpr}
For any $f \in {\mathcal D}(\L)$, the unnormalized filter  is the unique strong solution to the Zakai equation, for any $t \in [0,T]$
\begin{multline}\label{Zeq}
 \sigma_t (f) = f(\lambda_0) + \int_0^t \sigma_s (\L f) \ud s \\
+ \int_0^t  \int_0^{+ \infty} \Big ( \sigma_{s^-}( \lambda f(\lambda + \ell(z)) )   - \sigma_{s^-}(f) \Big ) (m^\a( \ud s, \ud z) - F^\a(\ud z) \ud s).
\end{multline}
\end{proposition}

\begin{proof}
First let us observe that
$\sigma_t (1) = \bE^\Q[ L_t  |  \H_t ] =  \frac{d \P}{d \Q}{\Big \vert  _{\mathcal H_t}}$, $t \in [0,T]$.
Thus the dynamics of $\sigma(1)$ can be easily obtained by considering the effect of the Girsanov change measure,  that is $\sigma(1)$ is the Dol\'eans-Dade exponential of the $(\Q, \bH)$-martingale $\int_0^t(\pi_{s^-} (\lambda) -1 )  (dN^\a_s  - \ud s)$

$$\sigma_t (1)  = \E \left( \int_0^t (\pi_{s^-} (\lambda) -1 )  (dN^\a_s  - \ud s) \right).$$

Hence  it solves
\begin{equation}\label{Z1}
 \ud \sigma_t (1) =  \sigma_{t^-}( 1) (\pi_{t^-} (\lambda) -1 ) (\ud N^\a_t - \ud t).
\end{equation}

By It\^o's formula we get that
$$ d\sigma_t (f) = \pi_{t^-}(f) \ud \sigma_t(1)  +\sigma_{t^-} (1) \ud  \pi_{t}(f) + d \Big ( \sum_{s \leq t} \Delta \pi_s(f) \Delta \sigma_s(1) \Big ).$$
Taking into account  Equations \eqref{KS} and \eqref{Z1} and that
$$d \Big (\sum_{s \leq t} \Delta \pi_s(f) \Delta \sigma_s(1) \Big ) =\int_0^{+ \infty} \sigma_{t^-} (1) ( \pi_{t^-}( \lambda) - 1)  \Big ( {  \pi_{t^-}( \lambda f(\lambda + \ell(z)) )  \over  \pi_{t^-}( \lambda)}  - \pi_{t^-}(f) \Big ) m^\a( \ud t, \ud z)$$
we get Equation \eqref{Zeq}.
Finally as in Theorem 4.7 in Ceci and Colaneri \cite{Ceci_Colaneri1} we can prove strong uniqueness for the Zakai equation by the strong uniqueness of the KS-equation.
\end{proof}

The Zakai  equation can be written also as
\begin{equation}
\ud \sigma_t (f) = [ \sigma_t (\widetilde \L f) -  \sigma_{t} ((\lambda  -1) f) ]\ud t
+  \int_0^{+ \infty} \Big ( \sigma_{t^-}( \lambda f(\lambda + \ell(z)) )   - \sigma_{t^-}(f) \Big ) m^\a( \ud s, \ud z) ,
\end{equation}
where the operator $\widetilde \L$ is defined in Equation \eqref{L tilde}
and as the  KS-equation it   has a natural recursive structure in terms of the sequence $\{T^\a_n\}_{n \geq 1}$.
Indeed, between two consecutive jump times, for $t \in [T^\a_n \wedge T, T^\a_{n+1} \wedge T)$ it reads as
\begin{equation} \label{Z1}
\ud \sigma_t (f) = [ \sigma_t (\widetilde \L f) -  \sigma_{t} ((\lambda  -1) f) ]\ud t
\end{equation}
and at a jump time $T^\a_n \leq T$
\begin{equation}\label{Zjump}\sigma_{T^\a_n}(f) = \sigma_{T^\a_{n^-}}( \lambda f(\lambda + \ell(Z^\a_n)).
\end{equation}

By the linear structure of the Zakai between consecutive jumps we get a convenient expression of the filter.

\begin{proposition}\label{lin}
The following representation holds, for any $f \in {\mathcal D}(\L)$ and $\forall n=1, 2, \dots$
\begin{equation}\label{rapp}
\pi_t(f) =  {\bE[ f(\widetilde \lambda^n_t) e^{-\int_{s}^t (\widetilde \lambda^n_u  -1) \ud u } ]|_{s = T^\a_{n-1}}  \over \bE[ e^{-\int_{s}^t (\widetilde \lambda^n_u  -1) \ud u } ]|_{s = T^\a_{n-1}}}, \quad t \in (T^\a_{n-1}\wedge T, T^\a_n\wedge T)
\end{equation}
where $\widetilde \lambda^n$ is the shot noise Cox process, solution $\forall t \in (T^\a_{n-1}\wedge T, T^\a_n\wedge T)$, of the SDE
\begin{equation}  \label{SN}
\ud \widetilde \lambda^n_t =  \alpha (\beta -  \widetilde \lambda^n_t ) \ud t  + \int_0^{+ \infty}  z m^\b(\ud t , \ud z),
\end{equation}
with initial law  $\pi_ {T^\a_{n-1}}$.
\end{proposition}
\begin{proof}
Let $\widetilde \lambda^{s,x}$ denotes the solution to Equation \eqref{SN} with initial condition $(s,x) \in [0,+ T) \times (0,+ \infty)$.
By It\^o's formula  $\forall s <t \leq T$
$$f(\widetilde \lambda^{s,x}_t) = f(x) + \int_s^t \widetilde \L f(\widetilde \lambda^{s,x}_u) \ud u +  M_t -M_s ,$$

with $M$ a $(\P, \bF)$-martingale.  Setting $\gamma_t =  e^{-\int_{s}^t (\widetilde \lambda^{s,x}_u  -1) \ud u }$ by the product rule  we obtain
$$ f(\widetilde \lambda^{s,x}_t) \gamma_t = f(x)  + \int_s^t  \widetilde \L f(\widetilde \lambda^{s,x}_u)  \gamma_u  \ud u -  \int_s^t  f(\widetilde \lambda^{s,x}_u)  (\widetilde \lambda^{s,x}_u -1) \gamma_u \ud u + \int_s^t \gamma_u \ud M_u$$
and, taking the expectation, we obtain
$$\bE[ f(\widetilde \lambda^{s,x}_t) \gamma_t ] = f(x)  + \int_s^t  \bE[\widetilde \L f(\widetilde \lambda^{s,x}_u)  \gamma_u ] \ud u -  \int_s^t  \bE[ f(\widetilde \lambda^{s,x}_u)  (\widetilde \lambda^{s,x}_u -1) \gamma_u ] \ud u.$$

Thus for any $f \in {\mathcal D}(\L)$, $\Psi_t(s,x)(f)  := \bE[ f(\widetilde \lambda^{s,x}_t) \gamma_t ]$  solves Equation  \eqref{Z1} and, as a consequence, ${\Psi_t(s,x)(f) \over \Psi_t(s,x)(1)}$ solves the KS-equation between two consecutive jump times given in Equation \eqref{ks1}.

Finally the statement follows by uniqueness of the KS-equation observing that
$$
\frac{\int_0^{+ \infty} \Psi_t(T^\a_{n-1},x)(f) \pi_{T^\a_{n-1} } (\ud x)}{\int_0^{+ \infty} \Psi_t(T^\a_{n-1},x)(1) \pi_{T^\a_{n-1} } (\ud x)}
$$
coincides with the filter at jump time  $T^\a_{n-1}$.
\end{proof}

\begin{remark}\label{Cshot}[Filtering of a  shot noise Cox process]
Taking $\beta =0$ and $\ell(z)=0$ in Equation \eqref{intensity} the claim arrival process $N^\a$ reduces to the Cox process with shot noise intensity considered in Dassios and Jang \cite{Dassios_Jang1}. Denoting by $\L^{SN}$  the Markov generator given by
$$
\L^{SN}  f(\lambda) =  - \alpha  \lambda f'( \lambda) +
\int_0^{+ \infty} [ f(\lambda + z) - f(\lambda) ]  \rho F^\b (\ud z),
$$
in this special case the KS and the Zakai equations are driven by $N^\a$ and are given by
\begin{equation}\label{KSshot}
 \ud \pi_t (f) = \pi_t (\L^{SN} f) \ud s +  \int_0^{+ \infty} \Big ( {\pi_{t^-}(  \lambda f)  \over
\pi_{t^-}(\lambda)} - \pi_{t^-}(f) \Big ) \left( \ud N^\a_t- \pi_{t^-} (\lambda)  \ud t \right),
\end{equation}
and
\begin{equation}
\ud \sigma_t (f) =  \sigma_t (\L^{SN} f) \ud t + \Big( \sigma_{t^-}( \lambda f )   - \sigma_{t^-}(f) \Big)\left ( \ud N^\a_t - \ud t \right),
\end{equation}
respectively.
In particular, the KS-equation between two consecutive jump times coincides with that in the general case in Equation \eqref{ks1} (with $\widetilde \L$ replaced by $\L^{SN}$) while the update at a jump time $T^\a_n$ (see Equation \eqref{ks1jump}) is given  by
\begin{equation}\label{ks1shot}
\pi_{T^\a_n}(f) = {\pi_{T^\a_{n^-}}( \lambda f) \over \pi_{T^\a_{n^-}}( \lambda)}.
\end{equation}
Analogously, the Zakai-equation between two consecutive jump times coincides with that in the general case in Equation \eqref{Z1} (with $\widetilde \L$ replaced by $\L^{SN}$),  while the update at a jump time $T^\a_n$ (see Equation \eqref{Zjump}) is given  by $\sigma_{T^\a_n}(f) = \sigma_{T^\a_{n^-}}( \lambda f)$.
\end{remark}


\section{The reduced optimal control problem under complete information}\label{Optimal_Control}

By the filtering techniques developed in Section \ref{sec:filtering}, the original problem under partial information is now reduced to a complete observation stochastic control problem, which involves only processes adapted or predictable w.r.t. the filtration $\bH$, under $\P$.
The $(\P,\bH)$-predictable projection measure  of  $m^{(1)}(\ud t, \ud z)$  (see Equation \eqref{eqn:m1m2}) associated with the loss process $C$  can be written in terms of the filter $\pi$:
$\pi_{t^-} (\lambda) F^\a(\ud z) \ud t$.
In the sequel we shall denote by $ \widetilde m^\a(\ud t, \ud z)$ the $(\P,\mathbb{H})$-compensated jump-measure
 \begin{equation} \label{cmp}
 \widetilde m^\a(\ud t, \ud z)= m^\a(\ud t, \ud z) - \pi_{t^-} (\lambda) F^\a(\ud z) \ud t.
 \end{equation}
We are now ready to state the analogous of Remark \ref{projection} in $(\P,\mathbb H)$:
\begin{remark}\label{rem:Hmg}
For any $\bH$-predictable random field $\{H(t,z), t \in [0,T], z \in [0, + \infty)\}$ and for $i=1,2$ the following equation holds:
$$
\bE \left[ \int_0^t \int_0^{+\infty} H(s,z)m^\a(\ud s, \ud z) \right] = \bE \left[ \int_0^t \int_0^{+\infty} H(s,z)\pi_{s^-} (\lambda) F^\a(\ud z) \ud s \right], \ t \in [0,T].
$$
Moreover, under the condition
$\bE \left[ \int_0^T \int_0^{+\infty} |H(s,z)| \pi_{s^-} (\lambda) F^\a(\ud z) \ud s \right] <+ \infty,$
the process
$$ \int_0^t \int_0^{+\infty} H(s,z) \widetilde m^\a(\ud s, \ud z), \quad t \in [0,T]$$
is a $(\P, \bH)$-martingale.
\end{remark}

The primary insurer wishes to subscribe a reinsurance contract to optimally control her wealth. The surplus process without reinsurance evolves according to the following equation:
\begin{equation}
\label{eqn:surplus_u=0}
dR_t = c_t\,dt - \int_0^{+\infty} z \,m^\a(\ud t,\ud z), \qquad R_0= R_0\in\mathbb{R}^+,
\end{equation}
where $\{c_t, t\in[0,T] \}$ denotes the insurance premium, which is assumed to be $\mathbb H$-predictable and such that $\mathbb E \left[ \int_0^T c_t dt \right]< +\infty$ and $R_0$ is the initial capital.
 The primary insurer subscribes a generic reinsurance contract, that is characterized by the retention function $\Phi$, which is an $\bH$-predictable random field, in general. We assume that the insurer can choose any reinsurance arrangement in a given class of admissible contracts, which is a family of functions of $z\in[0,+\infty)$ representing the retained loss. For practical applications, we suppose that the contracts are parametrized by a $n$-uple $u$ (the control) taking values in $U\subseteq\overline{\mathbb{R}}^n$, with $n\in\mathbb{N}$ and $\overline{\mathbb{R}}$ denoting the compactification of $\mathbb{R}$. 
Under an admissible strategy $u\in \mathcal{U}$ (the definition of admissibility set $\mathcal{U}$ will be given in Definition \ref{def:U}),
she retains the amount $\Phi(Z^\a_j,u_{T^\a_j})$ of the $j$-th claim, while the remaining $Z_j^{(1)} - \Phi(Z^\a_j,u_{T^\a_j})$ is paid by the reinsurer.

We suppose that $\Phi(z,u)$ is continuous in $u$ and there exist at least two points $u_N,u_M\in U$ such that
\[
0\le\Phi(z,u_M) \le \Phi(z,u) \le \Phi(z,u_N) = z \qquad \forall (z,u)\in[0,+\infty)\times U,
\]
so that $u=u_N$ corresponds to null reinsurance, while $u=u_M$ represents the maximum reinsurance protection. Notice that $u_M$ corresponds to full reinsurance when applicable.


\begin{example}\label{ex_reinsurance}
We can show how standard reinsurance contracts fit our model formulation.
\begin{enumerate}
\item Under proportional reinsurance, the insurer transfers a percentage $(1-u)$ of any future loss to the reinsurer, so we set
\[
\Phi (z,u) = uz, \qquad u \in [0,1].
\]
Selecting the scalar $u\in[0,1]=:U$ is equivalent to choosing the retention level of the contract. Notice that here $u_N=1$ means no reinsurance and $u_M=0$ is full reinsurance.
\item Under an excess-of-loss reinsurance policy, the reinsurer covers all the losses exceeding a retention level $u$, hence we fix the class of all the functions with this form:
\[
\Phi (z,u) = u \wedge z, \qquad u \in [0, +\infty ].
\]
So, here $U:=[0,+\infty]$, $u_N=+\infty$ and $u_M=0$ is full reinsurance.
\item Under a limited stop-loss reinsurance, for any claim the reinsurer covers the losses exceeding a threshold $u_1$, up to a maximum level $u_2>u_1$, so that the maximum loss is limited to $(u_2-u_1)$ on the reinsurer's side. In this case:
\[
\Phi (z,u) = z-  (z- u_1 )^{+} + (z- u_2 )^{+},
\]
so that $U=\{ (u_1,u_2): u_1\ge0, u_2\in[u_1,+\infty] \}$ and $u=(u_1,u_2)$. Clearly, we have that $u_M = (u_{M,1}, u_{M,2})=(0,+\infty)$ and $u_N$ can be any point on the line $u_1=u_2$.
A particular case is the so-called limited stop-loss with fixed reinsurance coverage, in which $u_2 = u_1 + \beta$, $\beta >0$. Here $U=[0,+\infty]$, $u_N=+ \infty$ and $u_M = 0$ corresponds to the maximum reinsurance coverage $\beta$.
\end{enumerate}
\end{example}


Clearly the insurer will have to pay a reinsurance premium $q^u=\{q^u_t, t\in[0,T] \}$, which depends on the strategy $u$. We assume that the reinsurance premium  admits the following representation:
\begin{equation}\label{eqn:q_of_u}
q^u_t (\omega) = q (t, \omega , u ) \quad \forall (t, \omega , u ) \in [0,T] \times \Omega \times U
\end{equation}
for a given function $q (t, \omega , u )\colon[0,T] \times \Omega \times U \rightarrow [0,+\infty) $ continuous in $u$, $\mathbb H$-predictable and with continuous partial derivatives $\frac{\partial q (t, \omega , u )}{\partial u_i} $, $i=1,\dots,n$.
We assume that, for any $ t\in [0,T] \times \Omega$
\[
q (t, \omega , u_N )= 0, \quad q (t, \omega , u ) \leq q (t, \omega , u_M ), \quad \forall u\in U,
\]
since a null protection is not expensive and the maximum reinsurance is the most expensive. 
In the following $q^u$ will denote the reinsurance premium associated with the dynamic reinsurance strategy $\{u_t, t \in [0,T] \}$. Notice that both insurance and reinsurance premia are assumed to be $\mathbb{H}$-predictable, since insurer and reinsurer share the same information.  Finally, we require the following integrability condition:
$$ \mathbb{E}\Big[ \int_0^T q^{u_M}_t \ud t \Big] <+ \infty,$$
which ensures that for any $u \in\mathcal{U}$,
$
\mathbb{E} \left[ \int_{0}^{T} q^u_s ds \right] <+ \infty.
$

\begin{example}[Expected value principle] \label{ex_EVP}
Under any admissible reinsurance strategy $u\in\mathcal{U}$, the expected cumulative losses covered by the reinsurer in the interval $[0,t]$ are given by
\[
 \mathbb{E}\left[ \int_0^t \int_0^{+\infty} ( z - \Phi(z,u_s)) \, m^\a( \ud s , \ud z) \right] = \mathbb{E}\left[ \int_0^t \int_0^{+\infty} ( z - \Phi(z,u_s)) \, \pi_{s^-} (\lambda) F^\a(\ud z) \ud s \right].
 \]
 According to the expected value principle, the premium $q^u$ applied by the reinsurer  has to satisfy
\[
 \begin{split}
\mathbb{E}\left[ \int_0^t  q_s^u \, \ud s \right] &= (1 + \theta_R) \mathbb{E}\left[ \int_0^t \int_0^{+\infty} ( z - \Phi(z,u_s)) \, \pi_{s^-} (\lambda) F^\a(\ud z) \ud s \right], \quad \forall u\in\mathcal{U}, \forall t \in [0,T],
\end{split}
\]
where $\theta_R>0$ denotes the safety loading applied by reinsurer.
Thus
\begin{equation} \label{ex_EVP1}
q_t^u = (1+ \theta _R ) \pi _{t^- }(\lambda) \int_0^{+\infty}   \left(  z - \Phi(z,u_t) \right) F^\a(\ud z).
\end{equation}

\end{example}

Summarizing, the surplus process with reinsurance evolves according to
\begin{equation}
\label{eqn:surplus}
dR^u_t = \left( c_t-q^u_t \right) dt - \int_0^{+\infty} \Phi(z, u_t) \, m^\a(dt,dz), \qquad R^u_0= R_0\in\mathbb{R}^+.
\end{equation}

Let us observe that
$$\int_0^t \int_0^{+\infty} \Phi(z, u_s) \widetilde m^\a(\ud s,\ud z), \quad t \in [0,T]$$
turns out to be a $(\P, \bH)$-martingale, because
\begin{equation*}
\begin{split}
\mathbb{E}\left[ \int_0^T \int_0^{+\infty} \Phi(z, u_s) \pi_{s^-} (\lambda) F^\a(\ud z) \ud s \right]
 &\le \mathbb{E}\left[ \int_0^T \int_0^{+\infty} z \, \pi_{s^-} (\lambda) F^\a(\ud z) \ud s \right] = \mathbb{E}\left[Z^\a\right] \mathbb{E} \left[ \int_0^T \lambda_s \ud s \right]
\end{split}
\end{equation*}\\
is finite, since Proposition \ref{mom} holds, and Remarks \ref{nuovo}, \ref{rem:Hmg} apply.\\

The insurance company invests its surplus in a risk-free asset with constant interest rate $r>0$, so that for any reinsurance strategy $u\in\mathcal{U}$ the wealth dynamics is
\begin{equation}
\label{eqn:X}
dX^u_t = dR^u_t + rX^u_t\,dt, \qquad X^u_0= R_0\in\mathbb{R}^+ ,
\end{equation}
whose solution is given by
\begin{equation}
\label{eqn:X_explicit}
X^u_t = R_0e^{rt} + \int_0^t e^{r(t-s)} \left( c_s-q^u_s \right) \,ds
-\int_0^t\int_0^{+\infty} e^{r(t-s)} \Phi(z, u_s)  \,m^{(1)}(ds,dz).
\end{equation}

As announced before, the insurer aims at optimally controlling her wealth using reinsurance. More formally, she aims at maximizing the expected exponential utility of terminal wealth, that is:
\[
\sup_{u\in\mathcal{U}}\mathbb{E}\bigl[ 1-e^{-\eta X^u_T} \bigr],
\]
which turns out trivially to be equivalent to the minimization problem:
\begin{equation}
\label{eqn:minpb}
\inf_{u\in\mathcal{U}}\mathbb{E}\bigl[ e^{-\eta X^u_T} \bigr],
\end{equation}
where $\eta>0$ denotes the insurer's risk aversion.

\begin{definition}
\label{def:U}
We define by $\mathcal{U}$ the class of admissible strategies, which are all the $U$-valued and $\mathbb{H}$-predictable processes, $\{u_t, t\in[0,T]\}$,  such that $\mathbb{E}\bigl[ e^{-\eta X^u_T} \bigr] < +\infty$.
Given $t \in [0,T]$, we will denote by $\mathcal{U}_t$ the class $\mathcal{U}$ restricted to the time interval $[t,T]$.
\end{definition}

Clearly, the admissible strategies must be $\mathbb{H}$-predictable, since they are based on the information at disposal.
The next assumptions are required in the sequel.
\begin{assumption}\label{ass_app_premium}
	We assume that for every $a >0$
	\begin{itemize}
	\item[i)]
	$ \mathbb E \left[ e^{a \ell(Z^\a)} \right] <+ \infty, \quad \mathbb E  \left[ e^{a Z^\a} \right] <+ \infty, \quad \mathbb E  \left[ e^{a Z^\b} \right] <+ \infty.$
	\item[ii)]
	$
	\mathbb E\left[ e^{a \int_0^T q^{u_M}_t\,dt} \right] <+ \infty.
	$
	\end{itemize}
\end{assumption}

\begin{lemma}
\label{lemma:expCfinito}
Under Assumption \ref{ass_app_premium} i) for every $a>0$ we have that $\mathbb{E}[e^{a C_T}] < +\infty$.
\end{lemma}
\begin{proof}
See Appendix \ref{app:useful_res}.
\end{proof}

\begin{remark}
Usually insurance companies apply a maximum policy $D>0$, i.e., they only repay claims up to the amount $D$ to the policyholders. In this setting, claims' sizes are of the form $min\{Z^\a_n, D\} \leq D$, hence condition $\mathbb E  \left[ e^{a Z^\a} \right] <+ \infty$ in Assumption \ref{ass_app_premium} is trivially satisfied.
\end{remark}

The class of admissible strategies is non empty, as shown by the next result.
\begin{proposition}
\label{prop:admissibleproc}
Under Assumption \ref{ass_app_premium}, every $\mathbb{H}$-predictable process $\{u_t, t\in[0,T]\}$ with values in $U$ is admissible.
\end{proposition}
\begin{proof}
Thanks to Lemma \ref{lemma:expCfinito}, the proof is basically the same as in Brachetta and Ceci \cite [Prop. 2.2, pag. 4]{Brachetta_Ceci_2020}.
\end{proof}


\section{The value process and its BSDE characterization}\label{Value_Process}

In this section we study the value process associated to the problem in Equation \eqref{eqn:minpb}. Let us introduce the Snell envelope for any $u\in\mathcal{U}$:
\begin{equation}
\label{eqn:W^u}
W^u_t = \essinf_{\bar{u}\in\mathcal{U}(t,u)}
{\mathbb{E}\biggl[e^{-\eta X^{\bar{u}}_T}\mid \mathcal{H}_t\biggr]} , \forall t \in [0,T]
\end{equation}
with $\mathcal{U}(t,u)$ defined, for an arbitrary control  $u\in\mathcal{U}$, as the restricted class of controls almost surely equal to $u$ over $[0,t]$
\begin{equation*}
\mathcal U(t,u):=\Big\{ \bar u \in \mathcal U: \bar{u}_s = u_s \ \text{a.s.} \ \text{for all} \ s\le t \le T  \Big\}.
\end{equation*}

Denoting by $\bar{X}^u_t=e^{-rt}X^u_t$ the discounted wealth:
\begin{equation}
\label{eqn:X_disc}
\bar X^u_t = R_0 + \int_0^t e^{-rs}\left( c_s-q^u_s \right) \,ds
-\int_0^t\int_0^{+\infty} e^{-rs} \Phi(z, u_s) \,m^\a(ds,dz),
\end{equation}
and introducing the value process as follows,
\begin{equation}
\label{eqn:V}
V_t = \essinf_{\bar{u}\in\mathcal{U}_t}{\mathbb{E}\biggl[e^{-\eta e^{rT}(\bar{X}^{\bar{u}}_T-\bar{X}^{\bar{u}}_t)}\mid \mathcal{H}_t\biggr]}, \forall t \in [0,T]
\end{equation}
(where $\mathcal{U}_t$ is introduced in Definition \ref{def:U}) we can show that $\forall u \in \mathcal{U}$
\begin{equation}\label{eqn:W^u_and_V}
W^u_t = e^{-\eta \bar{X}^u_t e^{rT}} V_t ,
\end{equation}
and, in turn, choosing null reinsurance, i.e. $u_t=u_N$, for any $t \in [0,T]$, we get
\begin{equation}
\label{eqn:VisWI}
V_t = e^{\eta \bar{X}^N_t e^{rT}} W^N_t, \forall t \in [0,T],
\end{equation}
where $\bar{X}^N$ and $W^N$ denote the discounted wealth and the Snell envelope in Equations \eqref{eqn:X_disc} and \eqref{eqn:W^u}, respectively, associated to null reinsurance.
Our aim is to develop a BSDE characterization for the process $\{W^N_t, t \in [0,T] \}$ which also provides a complete description of the value process $\{V_t, t \in [0,T]\}$ in Equation \eqref{eqn:V}.

The following definitions will play a key role for our BSDE characterization and its solution.
\begin{definition}
We define three classes of stochastic processes:
\begin{itemize}
\item $\mathcal{S}^2$ denotes the space of c\`adl\`ag $\mathbb{H}$-adapted processes $Y$ such that:
\[
\mathbb{E}[ (\sup_{t\in[0,T]} |Y_t|)^2] < +\infty.
\]
\item $\mathcal{L}^2$ denotes the space of c\`adl\`ag $\mathbb{H}$-adapted processes $Y$ such that:
\[
\mathbb{E} \left[ \int_0^T |Y_t|^2 dt \right] <+\infty.
\]
\item $\widehat{\mathcal{L}}^2$ denotes the space of $[0,+\infty)$-indexed $\mathbb{H}$-predictable random fields $\Theta=\{ \Theta_t(z), t \in [0,T], z \in [0, + \infty)\}$ such that:
\[
\mathbb{E} \left[ \int_0^T\int_0^{+\infty} \Theta_t^2(z) \pi_{t^-}(\lambda) F^\a(dz)  \,dt \right] <+\infty.
\]
\end{itemize}
\end{definition}

\begin{definition}
\label{def:spacegen}
We define
\[
\mathbb{M} = \{ (t,\omega,y,\theta(\cdot)): (t,\omega,y)\in[0,T]\times\Omega\times[0,+\infty) \text{ and } \theta\colon[0,+\infty)\to\mathbb{R} \ \textrm{measurable} \}.
\]
and, similarly, we denote by $\mathbb{M}^u$ the same set augmented with the variable \aC{$u\in U$}, i.e.,
\[
\mathbb{M}^u = \{ (t,\omega,y, \theta(\cdot),u): (t,\omega,y,u)\in[0,T]\times\Omega\times[0,+\infty)\times U \text{ and } \theta\colon[0,+\infty)\to\mathbb{R} \ \textrm{measurable} \}.
\]
\end{definition}

\begin{definition}
Let $\xi$ be an $\H_T$-measurable random variable. A solution to a BSDE  driven by the compensated random measure $\widetilde m^\a(\ud t, \ud z)$ given in Equation \eqref{cmp} and generator $g$ is a pair $(Y,  \Theta^Y) \in \mathcal{L}^2 \times \widehat{\mathcal{L}}^2$ such that
\[
Y_t = \xi + \int_t^T g( s, Y_s, \Theta^Y_s(\cdot)) \ud s  - \int_t^T \int_{0}^{+\infty} \Theta^Y_s (z) \widetilde m^\a(\ud s, \ud z), \quad t \in [0,T], \ \P-a.s.,
\]
where
$g(t, \omega , y, \theta (\cdot ))$ is a real-valued function on $\mathbb{M}$ 
which is $\mathbb{H}$-predictable w.r.t. $(t, \omega) \in [0,T] \times \Omega$.
\end{definition}

We first give some preliminary results.
\begin{proposition}\label{stima}
Under Assumption \ref{ass_app_premium} i), we have that
\begin{equation}\label{nuova}
 0< M^\a_t  \leq W^N_t \leq M^\b_t, \quad t \in [0,T],
	\end{equation}
where  $M^{(i)}$, $i =1,2$,  are the following $(\P,\mathbb{H})$-martingales
\begin{equation}\label{nuovaMG} M^\a_t = e^{-\eta R_0e^{rT}}
	\mathbb{E}\biggl[e^{-\eta \int_0^T e^{r(T-s)}c_s\,ds} \mid \mathcal{H}_t\biggr] ,  \quad M^\b_t = \mathbb{E}\biggl[e^{\eta e^{rT}C_T}\mid \mathcal{H}_t\biggr], \quad t \in [0,T].
\end{equation}
Moreover,
\begin{equation} \label{eqn:supfinito}
\mathbb{E}[ (\sup_{t\in[0,T]} W^N_t)^2] < +\infty.
\end{equation}
\end{proposition}
\begin{proof}
The discounted wealth in Equation \eqref{eqn:X_disc} for $u=u_N$ becomes
\[
\bar X^{N}_t = R_0 + \int_0^t e^{-rs}c_s\,ds
-\int_0^t\int_0^{+\infty} e^{-rs} z \,m^\a(ds,dz),
\]
hence Equation \eqref{eqn:V} implies that
$$
0 \le V_t \le \mathbb{E}\biggl[e^{-\eta e^{rT}(\bar{X}^N_T-\bar{X}^N_t)}\mid \mathcal{H}_t\biggr] \le \mathbb{E}\biggl[e^{\eta e^{rT}(C_T-C_t)}\mid \mathcal{H}_t\biggr] \qquad \mathbb{P}-\text{a.s.} \quad \forall t\in[0,T] .
$$
By Equation \eqref{eqn:VisWI}, we get that for any $t\in[0,T]$
\[
W^N_t  \le e^{-\eta \bar{X}^{N}_t e^{rT}}
	\mathbb{E}\biggl[e^{\eta e^{rT}(C_T-C_t)}\mid \mathcal{H}_t\biggr]
		\le \mathbb{E}\biggl[e^{\eta e^{rT}C_T}\mid \mathcal{H}_t\biggr]
		= M^\b_t, \quad \P-a.s.
\]
where $M^\b$ is a $(\P,\mathbb{H})$-martingale.
Moreover, we have that
\[
\begin{split}
W^N_t &= \essinf_{\bar{u}\in\mathcal{U}(t,u_N)}
{\mathbb{E}\biggl[e^{-\eta X^{\bar{u}}_T}\mid \mathcal{H}_t\biggr]} \geq  \mathbb{E}\biggl[e^{-\eta X^{N}_T}\mid \mathcal{H}_t\biggr]\\
	&\ge e^{-\eta R_0e^{rT}}
	{\mathbb{E}\biggl[e^{-\eta \int_0^T e^{r(T-s)}c_s\,ds} \mid \mathcal{H}_t\biggr]} = M^\a_t >0
	\quad \forall t \in [0,T].
\end{split}
\]
\\
To complete the proof, we observe that Doob's martingale inequality implies that
\[
\mathbb{E}[ (\sup_{t\in[0,T]} W^N_t)^2] \le \mathbb{E}[ (\sup_{t\in[0,T]} M^\b_t)^2]
	\le 4 \mathbb{E}[M_T^2] = 4\mathbb{E}\bigl[e^{2 \eta e^{rT}C_T}\bigr] <+\infty,
\]which is finite according to Lemma \ref{lemma:expCfinito}.
\end{proof}

\begin{proposition}[Bellman's Optimality Principle] \label{Bell}
Under Assumption \ref{ass_app_premium}
\begin{itemize}
	\item[i)] $\{W^u_t, t \in [0,T]\}$ in a $(\P,\bH)$-submartingale $\forall u \in \mathcal{U}$;
	
	\item[ii)] $\{W_t^{u^*}, t \in [0,T]\}$ in a $(\P,\bH)$-martingale if and only if $u^* \in \mathcal{U}$ is an optimal control.	\end{itemize}
	\end{proposition}
\begin{proof}
The proof follows the same lines of Brachetta and Ceci \cite[Proposition 3.2]{Brachetta_Ceci_2020}.
\end{proof}

\begin{remark}
By Proposition \ref{Bell} since $u = u_N \in \mathcal{U}$, $\{W^N_t, t \in [0,T]\}$ is a $(\P,\bH)$-submartingale and  $W^N \in \mathcal{S}^2 \subseteq \mathcal{L}^2$ (this follows from Proposition \ref{stima}). As a consequence, by Doob-Meyer decomposition and $(\P,\bH)$-martingale representation theorems, it admits the expression
$$W^N_t = \int_0^t \int_0^{+\infty} \Theta^{W^N}_s(z)  \ \widetilde m^\a (ds,dz) + A_t ,$$
where  $\Theta^{W^N} \in \widehat{\mathcal{L}}^2$  by \eqref{eqn:supfinito} and $\{A_t, t \in [0,T]\}$ in an increasing $(\P,\bH)$-predictable process such that $\mathbb{E} \left[ \int_0^T A^2_t  \ud t\right ] <+ \infty$.
Moreover, $W^N_T= e^{- \eta X^N_T}:= \xi$, and since the wealth associated to null reinsurance, $u=u_N$, is given by
$$X^N_T = R_0 e^{rT} + \int_0^T e^{r(T-t)} c_t \ud t -  \int_0^T \int_0^{+\infty} e^{r(T-t)} z m^\a(\ud t, \ud z) ,$$
we get the inequality $\xi  \leq e^{\eta e^{rT} C_T}$.
Thus Lemma \ref{lemma:expCfinito} guarantees that $\xi$  is a random variable with finite moments of any order.
Summarizing, we obtain that
\[
W^N_t = \xi - \int_t^T \int_0^{+ \infty} \Theta^{W^N}_s (z)  \ \widetilde m^\a (ds,dz) + \int_t^T \ud A_s.
\]
Next step provides an explicit expression for the process $A$ and characterizes $W^N$ and the optimal control via a BSDE approach.
\end{remark}

We now give the main result of this section.

\begin{theorem}\label{T1}
Under Assumption \ref{ass_app_premium},
$(W^N,\Theta^{W^N}) \in \mathcal{L}^2\times \widehat{\mathcal{L}}^2$ is the unique solution the following BSDE
\begin{equation}\label{bsde}
W^N_t = \xi - \int_t^T \int_0^{+ \infty} \Theta^{W^N}_s (z) \ \widetilde m^\a (ds,dz)
	- \int_t^T \esssup_{u \in \mathcal{U}}\widetilde{f} ( s, W^{N}_s , \Theta^{W^N}_s(\cdot ) , u_s )\,ds ,
\end{equation}
with terminal condition $\xi = e^{- \eta X^N_T}$, where
\begin{multline}\label{ftilde}
\widetilde{f} ( t, W^{N}_t , \Theta^{W^N}_t(\cdot ) , u_t ) = -W^N_{t-}\eta e^{r(T-t)} q^u_t \\
	- \int_0^{+\infty} [W^N_{t-}+\Theta^{W^N}_t (z) ] \big[ e^{-\eta e^{r(T-t)}(z-\Phi(z,u_t))} -1 \big]
	\pi _{t_{-} } (\lambda ) F^{(1)} (dz).
\end{multline}
Moreover, the process $u^* \in \mathcal{U}$ which satisfies
\begin{equation}\label{eqn:u*esssup}
\widetilde{f} ( t, W^{N}_t , \Theta^{W^N}_t(\cdot ) , u^*_t )
	= \esssup_{u \in \mathcal{U}}  \widetilde{f} ( t, W^{N}_t , \Theta^{W^N}_t(\cdot ) , u_t )
	\qquad \forall t\in[0,T]
\end{equation}
is an optimal control.
\end{theorem}

\begin{proof}
Theorem \ref{T1} follows directly by an existence result of a solution to the BSDE \eqref{bsde} (see Theorem \ref{ExUn} below) and a verification result (see Theorem \ref{VT1} below), which imply that any solution to the BSDE \eqref{bsde} coincides with the process $(W^N,\Theta^{W^N})$.
\end{proof}

\begin{remark}\label{RC}
Let us notice that
\begin{itemize}
	\item[i)] the driver of the BSDE \eqref{bsde} is always nonnegative, since via Equation \eqref{ftilde}, we get
\[
\esssup_{u \in \mathcal{U}} \widetilde{f} ( t, W^{N}_t , \Theta^{W^N}_t(\cdot ) , u_t ) \ge
	\widetilde{f} ( t, W^{N}_t , \Theta^{W^N}_t(\cdot ) , u_N ) = 0;
\]
	
	\item[ii)] there exists $u^* \in \mathcal{U}$ which satisfies Equation \eqref{eqn:u*esssup}: by hypothesis $q^u_t$ and
	$\Phi(z,u)$ are continuous on $u \in U$ and $U$ is compact, hence measurability selection results ensure that the maximizer is a $(\P,\bH)$-predictable process and Proposition \ref{prop:admissibleproc} holds.
	\end{itemize}
\end{remark}

\begin{theorem}\label{ExUn}
Under Assumption \ref{ass_app_premium}, there exists a unique solution $(Y,\Theta^{Y}) \in \mathcal{L}^2 \times \widehat{\mathcal{L}}^2$ to the BSDE \eqref{bsde}, i.e.,
\[
Y_t = \xi - \int_t^T \int_0^{+ \infty} \Theta^{Y}_s (z) \widetilde m^\a(\ud s, \ud z)
	+ \int_t^T f ( s, Y_s , \Theta^{Y}_s(\cdot) )\,ds ,
\]
with generator $f\colon\mathbb{M}\to[0,+\infty)$
\begin{eqnarray}
\label{eq:BSDE_gen}
f ( s, y, \theta( \cdot ) ) &=& -\esssup_{u \in \mathcal{U}}\widetilde{f} ( s, y , \theta( \cdot ), u_s ) = -\esssup_{u \in \mathcal{U}} \{ -y \eta e^{r(T-s)} q^u_s \nonumber \\
& & - \int_0^{+\infty} (y +\theta( z )) \big[ e^{-\eta e^{r(T-s)}(z-\Phi(z,u_s))} -1 \big]
	\pi _{s_{-} } (\lambda ) F^{(1)} (dz) \},
\end{eqnarray}
with $\mathbb{M}$ given in Definition \ref{def:spacegen}, and terminal condition $\xi = e^{- \eta X^N_T}$.
\end{theorem}

\begin{proof} The proof is postponed to Appendix \ref{AppendixC}.
\end{proof}

We now wish to provide a verification result.  To this end we recall the following result in Brachetta and Ceci \cite[Proposition 3.4]{Brachetta_Ceci_2020}.

\begin{proposition} \label{VT1}
Suppose there exists an $\mathbb{H}$-adapted process $D$ such that:
\begin{itemize}
\item $D=\{ D_t e^{ - \eta \bar{X}^u_t e^{rT}}, t \in [0, T] \}$ is an $(\P,\bH)$-sub-martingale for any $u \in \mathcal{U}$ and an $(\P,\bH)$-martingale for some $u^* \in \mathcal{U}$;
\item $D_T = 1$.
\end{itemize}
Then $D_t = V_t$ and $u^* $ is an optimal control.
\end{proposition}

\begin{theorem} \label{VT} (Verification Theorem)
Under Assumption \ref{ass_app_premium}, let $(Y,\Theta^{Y}) \in \mathcal{L}^2 \times \widehat{\mathcal{L}}^2$ be a solution to the BSDE \eqref{bsde} and  let $u^*\in {\mathcal U}$ be a process satisfying Equation \eqref{eqn:u*esssup}. Then $Y$ coincides with $W^N$,
\[
V_t = e^{\eta \bar{X}^N_t e^{rT}} Y_t \qquad \forall t \in [0,T],
\]
and  $u^*$ is an optimal control.
\end{theorem}
\begin{proof}
Let $(Y,\Theta^{Y}) \in \mathcal{L}^2 \times \widehat{\mathcal{L}}^2$ be a solution to the BSDE \eqref{bsde} and  $u^*\in {\mathcal U}$ be the process satisfying Equation \eqref{eqn:u*esssup} (see $ii)$ in Remark \ref{RC}). Define $D_t :=  e^{\eta \bar{X}^N_t e^{rT}} \ Y_t$, $t \in [0,T]$,
and observe that $D_T = e^{\eta X^N_T} \xi =1.$ We now prove that  $D=\{ D_t e^{ - \eta \bar{X}^u_t e^{rT}} , t \in [0, T] \}$ is a $(\P,\bH)$-sub-martingale for any $u \in \mathcal{U}$ and a $(\P,\bH)$-martingale for $u^*$. Then the statement will follow by Proposition \ref{VT1}.

By the product rule, for any $u \in \mathcal{U}$
\begin{align*}
\ud (D_t \ e^{ - \eta \bar{X}^u_t e^{rT}}) & =  \ud  ( e^{  \eta (\bar{X}^N_t - \bar{X}^u_t)e^{rT}} \ Y_t) \\
& =  e^{\eta (\bar{X}^N_{t^-}-\bar{X}^u_{t^-}) e^{rT}}\ \ud Y_t + Y_{t-} \ \ud ( e^{  \eta (\bar{X}^N_t - \bar{X}^u_t)e^{rT}} ) + \ud \Big ( \sum_{s \leq t} \Delta Y_s \ \Delta \big(e^{  \eta (\bar{X}^N_s - \bar{X}^u_s)e^{rT} }\big) \Big ).
\end{align*}
Recalling Equation \eqref{eqn:X_disc}, we notice that
\begin{equation}\label{C1}
\bar{X}^N_{t}-\bar{X}^u_{t} = \int_0^t e^{-rs}q^u_s\,ds
-\int_0^t\int_0^{+\infty} e^{-rs} (z-\Phi(z, u_s)) \,m^\a(ds,dz),
\end{equation}
and applying It\^o formula we obtain
\begin{align*} \ud ( e^{  \eta (\bar{X}^N_t - \bar{X}^u_t)e^{rT}} ) =  &
\ \eta e^{r T}  e^{  \eta (\bar{X}^N_t - \bar{X}^u_t)e^{rT}} e^{-rt} \ q^u_t \ \ud t \\
&+ e^{  \eta (\bar{X}^N_{t-} - \bar{X}^u_{t-})e^{rT}}  \int_0^{+\infty} \big(e^{-\eta e^{r(T-t)}(z-\Phi(z,u_t))}-1\big) m^\a(\ud t, \ud z).
\end{align*}
Finally, after some calculations we get, for any $u \in \mathcal{U}$
\begin{multline*}
\ud (D_t e^{ - \eta \bar{X}^u_t e^{rT}}) = \ud M^u_t +  e^{ \eta (\bar{X}^N_{t} - \bar{X}^u_{t})e^{rT}}
\biggl( \esssup_{w \in \mathcal{U}}  \widetilde{f} ( t, W^{Y}_t , \Theta^{W^Y}_t(\cdot ) , w_t ) -  \widetilde{f} ( t, W^{Y}_t ,\Theta^{W^Y}_t(\cdot ) , u_t ) \biggr),
\end{multline*}
where
\begin{multline*}
M^u_t = \int_0^t \int_0^{+\infty} e^{\eta (\bar{X}^N_{s^-}-\bar{X}^u_{s^-}) e^{rT}} \ \Theta^{W^Y}_s (z) \ e^{-\eta e^{r(T-s)}(z-\Phi(z,u_s))} \ \widetilde m^\a(\ud s, \ud z) \\
+\int_0^t \int_0^{+\infty} Y_{s-} \ e^{\eta (\bar{X}^N_{s^-}-\bar{X}^u_{s^-}) e^{rT}} \biggl(e^{-\eta e^{r(T-s)}(z-\Phi(z,u_s))}-1\biggr) \ \widetilde m^\a(\ud s, \ud z) , \quad t\in[0,T].
\end{multline*}
It remains to verify that, for any $u\in\mathcal{U}$, the process $\{M^u_t, t\in[0,T]\}$,
is a $(\P,\mathbb{H})$-martingale. To this end, it is sufficient to prove that the following two conditions hold
\begin{gather*}
\mathbb{E} \left[ \int_0^T \int_0^{+\infty}  e^{\eta (\bar{X}^N_{t}-\bar{X}^u_{t}) e^{rT}} \ \big|\Theta^{Y}_t (z) \big| e^{-\eta e^{r(T-t)}(z-\Phi(z,u_t))} \pi_{t} (\lambda) F^\a(\ud z) \ud t \right] < +\infty, \\
\mathbb{E} \left[ \int_0^T \int_0^{+\infty}  \ e^{\eta (\bar{X}^N_{t}-\bar{X}^u_{t}) e^{rT}} \ |Y_t| \big| e^{-\eta e^{r(T-t)}(z-\Phi(z,u_t))}-1 \big| \pi_{t} (\lambda) F^\a(\ud z) \ud t \right] < +\infty.
\end{gather*}

Using Equation \eqref{C1}, $\Phi(z,u_t) \leq z$, the well known inequality $2ab \leq a^2 + b^2$ $\forall a,b \in \R$ and Jensen's inequality, the first expectation above is dominated by
\[
\begin{split}
&\mathbb{E} \left[ e^{  \eta e^{rT} \int_0^T e^{-rt} q^{u_M}_t\,dt  } \int_0^T  \int_0^{+\infty} \bigl| \Theta^{Y}_t (z) \bigr| \pi_{t^-} (\lambda) F^\a(\ud z) \ud t \right]  \\
&\leq {1\over 2} \left \{ \mathbb{E} \left[ e^{  2\eta e^{rT} \int_0^T e^{-rt}q^{u_M}_t\,dt  } \int_0^T \pi_{t} (\lambda) \ud t \right] + \mathbb{E} \left[ \int_0^T \int_0^{+\infty} \bigl| \Theta^{Y}_t (z) \bigr|^2  \pi_{t} (\lambda) F^\a(\ud z) \ud t \right] \right \} \\
&\leq {1\over 4} \mathbb{E} \left[ e^{  4\eta e^{rT} \int_0^T e^{-rt}q^{u_M}_t\,dt  } \right] T
	+ {1\over 4} \mathbb{E} \left[ \int_0^T \pi^2_{t} (\lambda) \ud t\right]
	+ {1\over 2} \mathbb{E} \left[ \int_0^T \int_0^{+\infty} \bigl| \Theta^{Y}_t (z) \bigr|^2  \pi_{t} (\lambda) F^\a(\ud z) \ud t \right] \\
&\leq {1\over 4} \mathbb{E} \left[ e^{  4\eta e^{rT} \int_0^T e^{-rt}q^{u_M}_t\,dt  } \right] T
	+ {1\over 4} \mathbb{E} \left[ \int_0^T \pi_{t} (\lambda^2) \ud t\right]
	+ {1\over 2} \mathbb{E} \left[ \int_0^T \int_0^{+\infty} \bigl| \Theta^{Y}_t (z) \bigr|^2  \pi_{t} (\lambda) F^\a(\ud z) \ud t \right] \\
	&<+\infty,
\end{split}
\]
which is finite because of Assumption \ref{ass_app_premium} ii), Remark \ref{nuovo}, Proposition \ref{mom} and recalling that $\Theta^{Y} \in \widehat{\mathcal{L}}^2$. The second expectation is lower than
\[
\begin{split}
	&\mathbb{E} \left[ e^{  \eta e^{rT} \int_0^T e^{-rt}q^{u_M}_t\,dt  }
	\int_0^T |Y_{t}| \ \pi_{t} (\lambda) \ud t \right] \\
	&\le {1\over 2} \mathbb{E} \left[ \int_0^T \abs{Y_{t}}^2  \ud t \right]
	+ {1\over 4} \mathbb{E} \left[ e^{  4\eta e^{rT} \int_0^T e^{-rt}q^{u_M}_t\,dt  } \right] T
	+ {1\over 4} \mathbb{E} \left[ \int_0^T  \pi_{t}^4 (\lambda) \ud t \right] <+\infty ,
\end{split}
\]
where the first term is finite because $Y \in \mathcal{L}^2$, the second is finite by Assumption  \ref{ass_app_premium} $ii)$ and the third follows by Remark \ref{nuovo} and Proposition \ref{mom}.
\end{proof}


\section{The optimal reinsurance strategy}\label{Optimal_Reinsurance}

The aim of this section is to provide more insight into the structure of the optimal reinsurance strategy and investigate some special cases.\\
By Theorem \ref {T1}, $(W^N,\Theta^{W^N}) \in \mathcal{L}^2 \times \widehat{\mathcal{L}}^2$ is the unique solution to the BSDE \eqref{bsde} and any maximizer in Equation \eqref{eqn:u*esssup} provides an optimal control. Hence, exploiting the expression in Equation \eqref{eqn:q_of_u}, we look over $u \in \mathcal U$ for the maximizer of the function $\widetilde{f}\colon\mathbb{M}^u\to\mathbb{R}$ given by
\begin{multline}\label{eqn:ftilde_extended}
\widetilde{f} ( t, \omega, w, \theta  ( \cdot ), u ) =
	- w\eta e^{r(T-t)} q(t,\omega,u) \\
	- \int_0^{+\infty} (w+\theta(z)) (e^{-\eta e^{r(T-t)}(z-\Phi(z,u))} -1)
	\pi _{t_{-} } (\lambda )(\omega) F^{(1)} (dz) .
\end{multline}

The following general result provides a characterization of the optimal reinsurance strategy in the one-dimensional case, where $\Phi(z,u)$ is increasing in $u$, $u \in [u_M, u_N] \subset \overline{ \mathbb R}$.
In order to obtain some definite results we need to introduce a concavity hypothesis for the function $\widetilde{f}$ w.r.t. to the variable $u \in [u_M,u_N]$.

\begin{proposition}\label{prop:optreins_general}
Under Assumption \ref{ass_app_premium}, 
suppose that $\Phi (z, u)$ is differentiable in $u\in [u_M,u_N]$ for almost every $z \in (0, + \infty)$ and $\widetilde{f}$ given in Equation \eqref{eqn:ftilde_extended} is strictly concave in $u\in [u_M,u_N]$. Then the optimal reinsurance strategy is $u^*_t=\{ \hat{u}(t,W^N_{t^-},\Theta^{W^N}_t(\cdot)), t\in[0,T] \}$, where $\hat{u}$ is:
\begin{equation}\label{u*_general}
\hat{u}(t,\omega,w,\theta(\cdot)) =
\begin{cases}
u_M 	& (t,\omega,w,\theta(\cdot))\in R_0 \\
\bar{u}(t,\omega,w,\theta(\cdot))	& \mathbb{M}\backslash (R_0 \cup R_1) \\
u_N				& (t,\omega,w,\theta(\cdot))\in R_1,
\end{cases}
\end{equation}
and we define the two regions
\begin{align*}
R_0 &= \left\{ (t,\omega,w,\theta(\cdot))\in\mathbb{M} : \frac{\partial \widetilde{f} ( t, \omega, w, \theta  ( \cdot ), u_M )}{\partial u} <0 \right\} \\
 R_1 &=  \left\{ (t,\omega,w,\theta(\cdot))\in\mathbb{M} : \frac{\partial \widetilde{f} ( t, \omega, w, \theta  ( \cdot ), u_N )}{\partial u} >0 \right\},
\end{align*}
and $\bar{u}(t,\omega,w,\theta(\cdot)) \in (u_M,u_N)$ solves the following equation:
\begin{equation}\label{1st_order_gen}
- w \frac{\partial q(t,\omega,u)}{\partial u} = \int_{0}^{\infty} [w + \theta(z) ] z e^{-\eta e^{r (T-t)(z- \Phi (z, u))} } \frac{\partial \Phi(z,u)}{\partial u}  \pi _{t_{-} } (\lambda )(\omega) F^{(1)} (dz).
\end{equation}
\end{proposition}

\begin{proof}
We observe that $\widetilde{f}$ given in Equation \eqref{eqn:ftilde_extended}  is continuous and strictly concave in $u\in[u_M,u_N]$ by hypothesis. Hence the first order condition, which reads as Equation \eqref{1st_order_gen}, admits a unique solution $\bar{u}(t,\omega,w,\theta(\cdot))$ measurable  function on $\mathbb{M}$.
If we extend the function $\widetilde{f}$ to the whole real line, i.e. $u \in \R$, it is decreasing for $u<\bar{u}$ and increasing for $u>\bar{u}$, hence the maximizer on $[u_M,u_N]$ must be given by
\[
\hat{u}(t,\omega,w,\theta(\cdot)) = \max\{ u_M, \min\{\bar{u}(t,\omega,w,\theta(\cdot)), u_N \} \},
\]
which is equivalent to the Equation \eqref{u*_general}.
\end{proof}

\begin{remark}
If $q(t,\omega,u)$ and $\Phi (z, u)$ are linear or convex on $u \in  [u_M,u_N]$ then $\widetilde{f}$ is strictly concave in $u\in[u_M,u_N]$
and Proposition \ref{prop:optreins_general} applies.
\end{remark}

We now consider a few examples under the expected value principle for the reinsurance premium (see Remark \ref{ex_EVP}).

\subsection{Proportional Reinsurance}

In this subsection $\Phi(z,u) =z u$, $u \in [0,1]$.
According to Equation \eqref{ex_EVP1}, the reinsurance premium reads as:
\begin{equation}\label{eqn:evp}
q_t^u = (1+ \theta _R ) \mathbb{E} [Z^{(1)}] \pi _{t^- } (\lambda) (1-u_t), \quad \forall u \in \mathcal U.
\end{equation}

Notice that Assumption \ref{ass_app_premium} $ii)$ is automatically satisfied, since for every $a >0$ (see Appendix \ref{app:useful_res})
\[
	\mathbb E \left[ e^{a \int_0^T \pi_t(\lambda) \,dt} \right] < +\infty.
\]

\begin{proposition}
Under Assumption \ref{ass_app_premium} $i)$, there exist two stochastic thresholds $\theta^F_t < \theta^N_t$ such that 
\begin{equation}
\label{ustar}
u^*_t (\omega) =
\begin{cases}
	0  & \text{if } \theta_R < \theta^F_t(\omega)
	\\
	1 & \text{if } \theta_R > \theta^N_t(\omega)
	\\
	\bar{u}(t,\omega,W^N_{t^-}(\omega) ,\Theta^{W^N}_t(\cdot)(\omega))	& \text{otherwise,}
\end{cases}
\end{equation}
where
\begin{align*}
\theta^F_t &= {1 \over \mathbb{E} [Z^{(1)}] }
	\int_0^{\infty} \frac{W^N_{t^-}+ \Theta^{W^N}_t(z)}{W^N_{t^-}} z e^{-\eta e^{r (T-t) z} } F^{(1)} (dz) -1, \\
\theta^N_t &= {1 \over \mathbb{E} [Z^{(1)}] }
	\int_0^{\infty} \frac{W^N_{t^-} + \Theta^{W^N}_t(z)}{W^N_{t^-}} z  F^{(1)} (dz) -1
\end{align*}
and where $\bar{u}(t, w, \theta(\cdot)) \in (0,1)$ solves the following equation:
\begin{equation}
\label{eqn:evp_prop_stationary}
(1+ \theta _R ) \mathbb{E} [Z^{(1)}] = \int_0^{+\infty}  \frac{w+\theta(z)}{w} z e^{-\eta e^{r (T-t) z (1-u)} } F^{(1)} (dz).
\end{equation}
\end{proposition}
\begin{proof}
This follows immediately from Proposition \ref{prop:optreins_general}.
\end{proof}

Let us briefly comment the previous result. We can distinguish three cases, depending on the stochastic conditions (in particular, depending on the solution of the BSDE \eqref{bsde}):
\begin{itemize}
\item  if the reinsurer's safety loading $\theta_R$ is smaller than $\theta^F_t$, then  full reinsurance is optimal;
\item if  $\theta_R$ is larger than $\theta^N_t$, then null reinsurance is optimal and the contract is not subscribed;
\item lastly, if $\theta^F_t<\theta_R<\theta^N_t$, then the optimal retention level  takes values in $(0,1)$, that is, the ceding company transfers to the reinsurance a non null percentage of risk (not the full risk).
\end{itemize}

In other words, if the reinsurance contract is inexpensive, the full reinsurance is purchased. On the contrary, when the reinsurance cost is excessive, the primary insurer will retain all the risk. In the intermediate case $\theta^F_t<\theta_R<\theta^N_t$ the retention level takes values in the interval $(0,1)$. In any case, the concepts of inexpensive and expensive must be related to the underlying risk through the stochastic processes $W^N$ and $\Theta^{W^N}$, hence the thresholds are stochastic.

\subsection{Limited Stop-Loss Reinsurance}

The reinsurer's loss function is (see Example \ref{ex_reinsurance}$(3)$):
\begin{equation}
z - \Phi (z,u) =  z- \Phi(z,(u_1,u_2)) = (z- u_1 )^{+} - (z- u_2 )^{+} =
\left\{
\begin{array}{lcl}
0 & \textrm{if} & z \le u_1 \\
z - u_1 & \textrm{if} & z \in (u_1,u_2) \\
u_2- u_1 & \textrm{if} & z \ge u_2 , \\
\end{array}
\right.
\end{equation}
 with $u_1 < u_2$, so that the retention function is $\Phi (z,u) = z-  (z- u_1 )^{+} + (z- u_2 )^{+}$.

To obtain explicit results we will reduce our analysis to the case where the control is $u=u_1$, while $u_2=u_1+\beta$ is unequivocally determined, $\beta>0$ being the fixed maximum reinsurance coverage.

According to Equation \ref{ex_EVP1}, the expected value principle becomes
\begin{equation}\label{eq:q_stop_loss}
q_t^u =  (1+\theta_R) \pi_{t^-} (\lambda)  \int_{u_t}^{u_t+\beta} S_Z(z) dz , \quad \forall u \in \mathcal U,
\end{equation}
where $S_Z$ is the survival function $S_Z(z) = 1- F^{(1)} (z)$.

Let us observe that Assumption \ref{ass_app_premium} $ii)$ is automatically satisfied, in virtue of Lemma \ref{lem:finite}.

\begin{proposition}
\label{prop:u*LSL}
Under Assumption \ref{ass_app_premium} $i)$, there exists a stochastic threshold $\theta^L_t$ such that
\begin{equation}
\label{ustar_LSL}
u^*_t (\omega) =
\begin{cases}
	0  & \text{if } \theta_R < \theta^L_t(\omega)
	\\
	\bar{u}(t,\omega,W^N_{t^-}(\omega) ,\Theta^{W^N}_t(\cdot)(\omega))	& \text{otherwise,}
\end{cases}
\end{equation}
where
\[
\theta^L_t = {1 \over F^\a(\beta) }
	\int_0^{\beta} \frac{W^N_{t^-}+ \Theta^{W^N}_t(z)}{W^N_{t^-}} e^{-\eta e^{r (T-t) z} } F^{(1)} (dz) -1.
\]
and $\bar{u}(t, w, \theta(\cdot)) \in (0, +\infty)$ solves the following equation:
\begin{equation}
\label{eqn:evp_prop_stationaryLSL}
(1+ \theta _R ) \bigl(F^\a(u+\beta)-F^\a(u)\bigr) =
	\int_u^{u+\beta}  \frac{w+\theta(z)}{w} e^{-\eta e^{r (T-t) (z-u)} } F^{(1)} (dz).
\end{equation}
\end{proposition}
\begin{proof}

It is immediate to verify that $\widetilde f$ in Equation \eqref{eqn:ftilde_extended} is strictly concave in $u\in[0,+\infty)$, because the premium in Equation \eqref{eq:q_stop_loss} is convex in $u$ and $\frac{\partial \Phi(z, u )}{\partial u }=1$ for $z\in[u,u+\beta)$, while it is null elsewhere. The first order derivative is
\begin{eqnarray*}
\frac{\partial \widetilde{f} ( t, \omega, w, \theta  ( \cdot ), u )}{\partial u }
	&=& w\eta e^{r(T-t)} (1+\theta_R) \pi_{t^-} (\lambda)(\omega) [F^\a(u+\beta)-F^\a(u)] \\
	&-& \int_u^{u+\beta} (w+\theta(z)) \eta e^{r(T-t)} e^{-\eta e^{r(T-t)}(z-u)}
	\pi _{t_{-} } (\lambda )(\omega) F^{(1)} (dz).
\end{eqnarray*}
The maximizer is always finite (we can rule out the possibility of having null reinsurance, $u^*=+\infty$), while it is null if and only if $\frac{\partial \widetilde{f} ( t, \omega, w, \theta  ( \cdot ), 0 )}{\partial u} <0$, i.e., when $\theta_R < \theta^L_t(\omega)$. Conversely, if $\theta_R \ge \theta^L_t(\omega)$ the maximizer coincides with the unique stationary point satisfying $\frac{\partial \widetilde{f} ( t, \omega, w, \theta  ( \cdot ), u )}{\partial u} = 0$, which can be written as Equation \eqref{eqn:evp_prop_stationaryLSL}.
\end{proof}

Let us briefly comment the previous result. Differently from the proportional reinsurance, null reinsurance is never optimal and we can distinguish two cases, depending on the maximum coverage $\beta$ and the solution of the BSDE \eqref{bsde}:
\begin{itemize}
\item  if the reinsurer's safety loading $\theta_R$ is smaller than $\theta^L$ (i.e. the contract is inexpensive) then the maximum reinsurance coverage $\beta$ is optimal;
\item if  $\theta_R$ is larger than $\theta^L$ (i.e. the contract is inexpensive) then it is optimal purchasing reinsurance but not with maximum coverage.
\end{itemize}

\subsection{Excess of Loss Reinsurance}
The excess of loss contract, i.e., $z - \Phi (z,u) = (z- u )^{+} $ (see Example 4.2(2)) can be easily obtained from the previous case by letting $\beta \rightarrow \infty$. The optimal reinsurance strategy, under Assumption \ref{ass_app_premium} $i)$, becomes then:
\begin{equation}
\label{ustar_LSL}
u^*_t (\omega) =
\begin{cases}
	0  & \text{if } \theta_R < \theta^L_t(\omega)
	\\
	\bar{u}(t,\omega,W^N_{t^-}(\omega) ,\Theta^{W^N}_t(\cdot)(\omega))	& \text{otherwise,}
\end{cases}
\end{equation}
where
\[
\theta^L_t = \int_0^{+\infty} \frac{W^N_{t^-}+ \Theta^{W^N}_t(z)}{W^N_{t^-}} e^{-\eta e^{r (T-t) z} } F^{(1)} (dz) -1
\]
and $\bar{u}(t, w, \theta(\cdot)) \in (0, + \infty)$ solves the following equation:
\begin{equation}
\label{eqn:evp_prop_stationaryLSL}
(1+ \theta _R ) S_Z(u) =
	\int_u^{+\infty}  \frac{w+\theta(z)}{w} e^{-\eta e^{r (T-t) (z-u)} } F^{(1)} (dz).
\end{equation}

As in the Limited Stop-Loss Reinsurance case, null reinsurance is never optimal and two cases are possible, depending on the solution of the BSDE \eqref{bsde}:
\begin{itemize}
\item when the reinsurance contract is inexpensive ($\theta_R<\theta^L$), the full reinsurance is optimal;
\item otherwise, it becomes optimal to purchase an intermediate protection level.
\end{itemize}

\bigskip

\textbf{Acknowledgements:}
The first and third authors have been partially supported by the Project INdAM-GNAMPA, number: U-UFMBAZ-2020-000791.
The first three authors  have been partially supported by the Project INdAM-GNAMPA, number: U-UFMBAZ-2022-000765.\\
All the authors have been  partially supported by University of Padova Grant BIRD 190200/19.\\
The authors thank Elena Bandini, Fulvia Confortola, Andrea Cosso and Paolo Di Tella for some useful indications about the state-of-the-art on existence and uniqueness results for BSDEs.\\
The authors thank two anonymous Referees and the Associate Editor for some relevant comments and suggestions.


\appendix

\section{Proofs of auxiliary results}\label{Appendix}
\label{appendix:proofs}

\begin{lemma}
\label{lemma:exp_predictable}
Let $(\Omega,\F, \P;\bF)$ be a filtered probability space and assume that the filtration $\bF=\{\F_t, \ t \in [0,T]\}$ satisfies the usual hypotheses.  Let $N$ be a standard Poisson process with $\bF$-intensity $\lambda>0$ and let $\{b_t, t\in [0,T]\}$ an $\mathbb{F}$-predictable process. Then
$$
\mathbb E \left[ e^{\int_0^T b_t\,dN_t }\right]
	= \mathbb E \left[ e^{  \int_0^T(e^{b_t}-1)\lambda\,dt }\right],
$$
provided that the last expectation is finite.
\end{lemma}
\begin{proof}
In order to show that the statement is valid for any bounded $\mathbb{F}$-predictable process, see Br\'emaud \cite[T4 Theorem, Appendix A1]{Bremaud}, it is sufficient to prove our result for any arbitrary process
\[
b_t = \mathbbm{1}_{(t_1,t_2]}(t) \mathbbm{1}_A, \quad 0\le t_1<t_2\le T, \quad  A\in\mathcal{F}_{t_1}.
\]
Let $0\le t_1<t_2\le T, A\in\mathcal{F}_{t_1}$ and denote by $A^C$ the complementary set of $A$.
Then we have that
\[
\begin{split}
\mathbb E \left[ e^{\int_0^T b_t\,dN_t }\right]
	&= \mathbb E \left[ e^{\int_{t_1}^{t_2} \mathbbm{1}_A\,dN_t }\right] = \mathbb E \left[ e^{ (N_{t_2}-N_{t_1}) \mathbbm{1}_A }
	(\mathbbm{1}_A+\mathbbm{1}_{A^C}) \right] \\
	&= \mathbb E \left[ \mathbb{E}[e^{ (N_{t_2}-N_{t_1})} \mid \mathcal{F}_{t_1}]\mathbbm{1}_A
		+\mathbbm{1}_{A^C} \right] \\
	&= \mathbb E \left[ \mathbb{E}[e^{ (N_{t_2}-N_{t_1})}]\mathbbm{1}_A
		+\mathbbm{1}_{A^C} \right] .
\end{split}
\]
Now the inner expectation corresponds to the Laplace transform of a Poisson random variable, since $(N_{t_2}-N_{t_1}) \sim \textrm{Po}\left( \lambda (t_2 - t_1) \right)$, namely
$ \mathbb{E}[e^{ (N_{t_2}-N_{t_1})}] = e^{(e-1)\lambda(t_2-t_1)}$.
Substituting and rearranging the terms we then get that
\[
\mathbb E \left[ e^{\int_0^T b_t\,dN_t }\right]
	= \mathbb E \left[ e^{(e-1)\lambda(t_2-t_1)\mathbbm{1}_A} \right] .
\]
On the other hand, we notice that
\[
e^{b_s}-1 =  e^{ \mathbbm{1}_{(t_1,t_2]}(t) \mathbbm{1}_A} - 1 =  e \cdot \mathbbm{1}_{(t_1,t_2]}(t) \mathbbm{1}_A - \mathbbm{1}_{(t_1,t_2]}(t) \mathbbm{1}_A = (e-1)\mathbbm{1}_{(t_1,t_2]}(t) \mathbbm{1}_A
\]
and so
\[
\mathbb E \left[ e^{  \int_0^T(e^{b_t}-1)\lambda\,dt }\right] = \mathbb E \left[ e^{  \int_0^T (e-1)\mathbbm{1}_{(t_1,t_2]}(t) \mathbbm{1}_A \lambda\,dt }\right] = \mathbb E \left[ e^{(e-1)\lambda(t_2-t_1)\mathbbm{1}_A} \right] ,
\]
which proves the statement for any bounded $\mathbb{F}$-predictable process. To complete the proof, we extend this result to unbounded processes.
Assume that $\{b_t, t\ge0\}$ is an arbitrary $\mathbb{F}$-predictable process and define a sequence of $\mathbb F$-stopping times $\tau_n = \inf\{ t\ge0: b_t > n \}, \quad n\ge1$.
Clearly, $\tau_n\to+\infty$ as $n\to+\infty$. By the first part of the proof, we know that
$$
\mathbb E \left[ e^{\int_0^{T\land\tau_n} b_t\,dN_t }\right]
	= \mathbb E \left[ e^{  \int_0^{T\land\tau_n} (e^{b_t}-1)\lambda\,dt }\right],
$$
so that, to complete the proof, it remains to pass to the limit $n\to+\infty$ and to apply the monotone convergence theorem to the family of random variables $X_n := e^{\int_0^{T\land\tau_n} b_t\,dN_t} , n \ge 1,$ in the case when $b$ is positive, or to $\overline X_n := \frac{ e^{\int_0^{T\land\tau_n} b_t^+ \,dN_t} }{e^{\int_0^{T\land\tau_n} b_t^- \,dN_t}  }, n \ge 1$ for a general $b$.
\end{proof}

\begin{lemma}\label{lemma:exp_predictable2}
Let $(\Omega,\F, \P;\bF)$ be a filtered probability space and assume that the filtration $\bF=\{\F_t, \ t \in [0,T]\}$ satisfies the usual hypotheses. Let $N(dt,dz)$ be a Poisson random measure on $[0,T]\times[0,+\infty)$ with $\bF$-intensity kernel $\lambda F(dz)\,dt$. Then for any $\bF$-predictable and $[0,+\infty)$-indexed process $\{H(t,z), t \in [0;T], z \in [0,+\infty) \}$ we have that
$$
\mathbb E \left[ e^{\int_0^T\int_0^{+\infty} H(t,z)\,N(dt,dz) }\right]
	= \mathbb E \left[ e^{\int_0^T\int_0^{+\infty} ( e^{H(t,z)} -1) \lambda F(dz)\,dt }\right] ,
$$
provided that the last expectation is finite.
\end{lemma}
\begin{proof}
It is sufficient to prove the result for any process $\{H(t,z), t \in [0;T], z \in [0,+\infty) \}$ of this form:
\[
H(t,z) = b_t \mathbbm{1}_A, \quad \forall t\ge0,  A\in\mathcal{B}([0,+\infty)),
\]
where $b_t$ is $\mathbb{F}$-predictable and $\mathcal{B}([0,+\infty))$ denotes the Borel $\sigma$-algebra of subsets of $[0,+\infty)$.
By Lemma \ref{lemma:exp_predictable} we readily obtain that
\[
\begin{split}
\mathbb E \left[ e^{\int_0^T H(t,z)\,N(dt,dz) }\right]
&= \mathbb E \left[ e^{  \int_0^T b_t N(dt,A)}\right] = \mathbb E \left[ e^{  \int_0^T(e^{b_t}-1)\int_A F(dz)\,\lambda\,dt }\right] = \mathbb E \left[ e^{  \int_0^T\int_0^{+\infty}
	(e^{b_t}-1)\mathbbm{1}_{A}(z)\,F(dz)\,\lambda\,dt }\right] \\
	&= \mathbb E \left[ e^{  \int_0^T\int_0^{+\infty}
	(e^{b_t \mathbbm{1}_{A}(z)}-1) \,F(dz)\,\lambda\,dt }\right] = \mathbb E \left[ e^{  \int_0^T\int_0^{+\infty}
	(e^{H(t,z)}-1) \,F(dz)\,\lambda\,dt }\right],
\end{split}
\]
where we have used that $N((0,t] \times A)$ is a Poisson process with intensity $\int_A F(dz)\,\lambda$.
\end{proof}


\section{Proof of key Lemmas} \label{app:useful_res}

We focus here on the finiteness of $ \mathbb E \left[  e^{a N_T^{(1)}} \right]$,  $ \mathbb E \left[  e^{a \int_0^T \lambda_s ds} \right]$, $\mathbb E \left[  e^{a \int_0^T \pi_s(\lambda) ds} \right]$ and $\mathbb{E}[e^{a C_T}]$, which are computed under $\P$ for an arbitrary real constant $a> 0$. Here $N^{(1)}$ is a standard Poisson process under  $(\Q, \bF)$ and a counting process with intensity $\lambda$ (given in Equation \eqref{intensity}) under $(\P, \bF)$. We will exploit the measure change introduced in detail in Section \ref{sec:model} and we will work under Assumption \ref{ass_app_premium} $i)$.
We prove the following:
\begin{lemma}\label{lem:finite}
Under Assumption \ref{ass_app_premium} $i)$
$$
\mathbb E \left[  e^{a N_T^{(1)}} \right] < + \infty   \quad  \mathbb E \left[  e^{a \int_0^T \lambda_s ds} \right]  <+ \infty, \quad  \textrm{and} \quad
\mathbb E \left[  e^{a \int_0^T \pi_s(\lambda) ds} \right] <+ \infty.$$
\end{lemma}
\begin{proof}
First of all, we show that under Assumption \ref{ass_app_premium} $i)$ we have
\begin{equation}\label{eq:finite_exp_Q_intL}
\mathbb E^{\Q} \left[  e^{a \int_0^T \lambda_s ds} \right]  <+ \infty.
\end{equation}
Recalling Equation \eqref{eq:lambda_bound}, for a suitable $c_1 >0$ and for $c_2 = aT$ we find that
\begin{eqnarray*}
\mathbb E^{\Q} \left[  e^{a \int_0^T \lambda_s ds} \right]
	& \le &  \mathbb E^{\Q} \left[  e^{a T \left(  \max\{\lambda_ 0,\beta\} +  \sum_{j=1}^{N^\a_T}  \ell(Z^\a_j) +  \sum_{j=1}^{N^\b_T}  Z^\b_j \right)} \right] \le  c_1 \mathbb E^{\Q} \left[  e^{c_2 \left( \sum_{j=1}^{N^\a_T}  \ell(Z^\a_j) + \sum_{j=1}^{N^\b_T}  Z^\b_j \right)} \right] \\
	&= & c_1  e^{ T \big (\mathbb E^{\Q}   [ e^{  c_2 \ell(Z^\a)  }  ] - 1 \big ) } e^{ T \big (\mathbb E^{\Q}   [ e^{ c_2 Z^\b  }   ] - 1 \big ) }  <+ \infty,
\end{eqnarray*}
where we used the mutual independence of $N^{(1)}, N^{(2)}, {\{ Z_{n}^{(1)} \}}_{n \ge 1}, {\{ Z_{n}^{(2)} \}}_{n \ge 1}$ under $\Q$ and, in the last equality, we followed the path traced in the proof of Proposition \ref{Prop:L_mart}. Finally Assumption \ref{ass_app_premium} $i)$ gives the finiteness of the expectation under $\Q$.

To prove that $\mathbb E \left[  e^{a N_T^{(1)}} \right]$ is finite we exploit the change of measure from $\P$ to $\Q$ via $\frac{d \P}{d \Q}_{\vert  _{\mathcal F_T}} = L_T$, with $L_T$ given in Equation \eqref{eqn:L}, so that
\begin{eqnarray*}
\mathbb E \left[  e^{a N_T^{(1)}} \right] & = &  \mathbb E^{\Q} \left[ L_T  e^{a N_T^{(1)}} \right] = \mathbb E^{\Q} \left[  e^{ -\int_0^T (\lambda_s - 1) ds  + \int_0^T ( \ln (\lambda_{s^-}) + a )  d N_s^{(1)} }  \right] \le C \ \mathbb E^{\Q} \left[  e^{\int_0^T ( \ln (\lambda_{s^-}) + a )  d N_s^{(1)} }  \right]
\end{eqnarray*}
for a suitable constant $C>0$.
Now we recall that under $\Q$ the Poisson process $N^{(1)}$ has unitary intensity and for any predictable process $b$ we have that $\label{utile} E^{\Q} \left[  e^{\int_0^T  b_s  d N_s^{(1)} }  \right] = E^{\Q} \left[  e^{\int_0^T  (e^{b_s}  -1 ) ds}  \right]$, according to Lemma \ref{lemma:exp_predictable}.
Hence, taking $b_s= \ln (\lambda_{s^-}) + a$, we obtain
\begin{equation}\label{N1}
\mathbb E \left[  e^{a N_T^{(1)}} \right]
	\le C \mathbb E^{\Q} \left[  e^{\int_0^T ( \ln (\lambda_{s^-}) + a )  d N_s^{(1)} }  \right]
	=  C \ \mathbb E^{\Q} \left[  e^{\int_0^T  (e^a \lambda_s - 1)  ds} \right] <+ \infty,
\end{equation}
which is finite because of Equation \eqref{eq:finite_exp_Q_intL}.

We show now that $\mathbb E \left[  e^{a \int_0^T \lambda_s ds} \right]  <+ \infty$ $\forall a>0$. We proceed as above: passing under $\Q$ via $L_T$, recalling Equation \ref{eq:lambda_bound} and introducing the integer-valued random measure $m^\a(\ud t, \ud z)$, we find 
\begin{eqnarray*}
\mathbb E \left[  e^{a \int_0^T \lambda_s ds} \right] & = &  \mathbb E^{\Q} \left[ L_T  e^{a \int_0^T \lambda_s ds} \right] =  \mathbb E^{\Q} \left[  e^{  \int_0^T [ (a-1) \lambda_s +1] ds  + \int_0^T  \ln (\lambda_{s^-}) d N_s^{(1)} }  \right] \\
& \le & C_1 \ \mathbb E^{\Q} \left[   e^{C_2 \left( \sum_{j=1}^{N^\a_T}  \ell(Z^\a_j) + \sum_{j=1}^{N^\b_T}  Z^\b_j \right)}  e^{  \int_0^T  \ln (\lambda_{s^-}) d N_s^{(1)} } \right] \\
&=&  C_1 \ \mathbb E^{\Q} \left[   e^{C_2  \sum_{j=1}^{N^\b_T}  Z^\b_j }   \right ] \mathbb E^{\Q} \left[   e^{\int_0^T \int_0^{+\infty} [C_2 \ell(z) + \ln (\lambda_{s^-})] m^\a(\ud t, \ud z)  } \right]
\end{eqnarray*}
for a suitable constant $C_1 >0$.
We now apply Lemma \ref{lemma:exp_predictable2} under $\Q$ and for $H(t,z) =  [C_2 \ell(z) + \ln (\lambda_{t^-})]$ and with $\nu^{\a, \Q}(\ud t, \ud z)=  F^\a(\ud z) \ud t$ and we get:
\begin{eqnarray*}
\mathbb E \left[  e^{a \int_0^T \lambda_s ds} \right] & \le &  C_1 \ \mathbb E^{\Q} \left[   e^{C_2  \sum_{j=1}^{N^\b_T}  Z^\b_j }   \right ] \mathbb E^{\Q} \left[   e^{\int_0^T \int_0^{+\infty} [C_2 \ell(z) + \ln (\lambda_{s^-})] m^\a(\ud t, \ud z)  } \right]   \\
&=&  C_1 \ \mathbb E^{\Q} \left[   e^{C_2  \sum_{j=1}^{N^\b_T}  Z^\b_j }   \right ] \mathbb E^{\Q} \left[   e^{\int_0^T \int_0^{+\infty} ( \lambda_{s} e^{C_2 \ell(z)} - 1 )  F^\a(\ud z) \ud s } \right]  \\
&=& C_1 \ \mathbb E^{\Q} \left[   e^{C_2  \sum_{j=1}^{N^\b_T}  Z^\b_j }   \right ] \mathbb E^{\Q} \left[   e^{\int_0^T  \left( \lambda_{s} \mathbb E^{\Q}[ e^{C_2 \ell(Z_1^{(1)})}] - 1 \right)   \ud s } \right]
\end{eqnarray*}
which is finite under Assumption \ref{ass_app_premium} $i)$.

It remains to prove that $\mathbb E \left[  e^{a \int_0^T \pi_s{(\lambda)} ds} \right] <+ \infty$ $\forall a>0$. The structure of the filtering equation implies that over  $[0,T]$ the filter attains its maximum value at a jump time.
More precisely, we showed in  Remark \ref{ultimo} that the filter is dominated by a process with exponential decay behavior between two consecutive jumps, hence the maximum over  $[0,T]$ is attained at a jump time $\tau \leq T$ such that
\[
\pi_\tau(\lambda) = \max \left\{ \pi_{T^\a_1}(\lambda), \dots ,
	\pi_{T^\a_{N_T^\a}}(\lambda) \right\}.
\]
Notice that the maximum is taken over a finite number of elements, because the jump process $N^\a$ is non explosive. Then, using Jensen's inequality we have that
\[
\mathbb E \left[  e^{a \int_0^T \pi_t(\lambda)\, dt} \right] \le \mathbb E \left[  e^{a T \pi_\tau(\lambda)} \right] \le \mathbb E \left[ \pi_\tau( e^{a T \lambda} )  \right] = \mathbb E \left[ e^{a T \lambda_\tau}  \right] <+\infty.
\]

The last inequality is implied by the fact that $\tau \leq T$ and so the following inequalities hold
\[
\begin{split}
\mathbb E \left[ e^{a T \lambda_\tau}  \right] & =  \mathbb E^{\Q} \left[ L_T  e^{a T \lambda_\tau}  \right]  \le C_1 \mathbb E^{\Q} \left[   e^{a T \lambda_\tau}  e^{ \int_0^T  \ln (\lambda_{s^-}) d N_s^{(1)} }  \right] \\
& \le C_1 \ \mathbb E^{\Q} \left[   e^{C_2 \left( \sum_{j=1}^{N^\a_T}  \ell(Z^\a_j) + \sum_{j=1}^{N^\b_T}  Z^\b_j \right)}  e^{  \int_0^T  \ln (\lambda_{s^-}) d N_s^{(1)} } \right] \end{split}
\]
for suitable constants $C_i>0$, $i=1,2$, and we can prove the finiteness by doing the same computations to prove that  $\mathbb E \left[  e^{a \int_0^T \lambda_s ds} \right]  < + \infty$.
\end{proof}

Based on the previous Lemma, we conclude this section proving the useful result given in Lemma \ref{lemma:expCfinito}, i.e. for every $a>0$
\[
\mathbb{E}[e^{a C_T}] < +\infty.
\]
\begin{proof}[\textbf{Proof of Lemma \ref{lemma:expCfinito}}]
We have that for a suitable constant $\kappa>0$, passing under $\Q$ via the Radon-Nikodym derivative $L_T$ given in Equation \eqref{eqn:L} and using Lemma \ref{lemma:exp_predictable2},
\begin{align*}
\mathbb{E}[e^{a C_T}]
	&= \mathbb E^{\Q} \left[ e^{-\int_0^T(\lambda_t-1)\,dt + \int_0^T \ln \lambda_{t-}\,dN^\a_t}
		e^{\int_0^T\int_0^{+\infty} az\,m^\a(dt,dz)} \right] \le \kappa \, \mathbb E^{\Q} \left[ e^{ \int_0^T\int_0^{+\infty}(\ln \lambda_{t-}+az)\,m^\a(dt,dz) } \right] \\
	&= \kappa \, \mathbb E^{\Q} \left[ e^{ \int_0^T\int_0^{+\infty}( e^{\ln \lambda_{t-}+az}-1)\,F^\a(dz)\,dt } \right] = \kappa\, \mathbb E^{\Q} \left[ e^{ \int_0^T \lambda_{t-} ( \mathbb{E}[e^{aZ^\a}] -1)\,dt } \right] <+\infty,
\end{align*}
where the finiteness comes from Equation \eqref{eq:finite_exp_Q_intL} and Assumption \ref{ass_app_premium} $i)$.
\end{proof}


\section{Proof of Theorem \ref{ExUn}}\label{AppendixC}
\begin{proof}
In order to apply Papapantoleon, Possamai and Saplaouras \cite[Theorem 3.5]{Papa_Possa_Sapla2018} we start by verifying that the BSDE data are standard under $\widehat \beta$, i.e., that assumptions $\mathbf{(F1)} - \mathbf{(F5)}$ therein are satisfied for a $\widehat \beta >0$. We will show that in our setting any $\widehat \beta >0$ works fine (see $\mathbf{(F4)}$ below).
\begin{itemize}
\item[$\mathbf{(F1)}$] The process $\{\widetilde{C}_t, t\in[0,T]\}$, with
$\widetilde{C}_t = \int_0^t\int_0^{+\infty} z \widetilde m^\a(\ud s, \ud z$
is a $(\P, \bH)$-martingale because of Remark \ref{rem:Hmg}. Notice that $\widetilde{C}$  is a pure-jump martingale, since the Brownian part is absent. Moreover,
\[
\begin{split}
\mathbb{E}[ \widetilde{C}_t^2] = \mathbb{E}\biggl[ \int_0^t\int_0^{+\infty} z^2 \pi _{s_{-} } (\lambda ) F^{(1)} (dz) \biggr] =  \mathbb{E} \left[  (Z^\a)^2 \right]  \mathbb{E} \left [\int_0^t  \pi _{s_{-} } (\lambda ) \ud s \right ],
\end{split}
\]
which is finite for every $t\in[0,T]$ according to Remark \ref{nuovo}. Hence $\sup_{t\in[0,T]} \mathbb{E}[ \widetilde{C}_t^2] <+\infty$ and Papapantoleon, A., Possamai, D. and Saplaouras \cite[Assumption 2.10]{Papa_Possa_Sapla2018} is satisfied.  In particular, the disintegration property is fulfilled with the transition kernel $K^{\omega}$ on $(\Omega\times[0,T], \mathcal{P})$ (here $\mathcal{P}$ denotes the $\mathbb{H}$-predictable sigma-field on $\Omega\times[0,T]$)
\begin{equation}\label{eq:trans_ker}
K_t^{\omega}(dz) = \pi_{t^-} (\lambda) F^\a(\ud z) .
\end{equation}
\item[$\mathbf{(F2)}$] Lemma \ref{lemma:expCfinito} guarantees that the terminal condition $\xi = e^{- \eta X^N_T}$ has finite moments of any order. See also $\mathbf{(F4)}$ below for additional details.
\item[$\mathbf{(F3)}$] 
We need to prove that the generator  $f$  satisfies a stochastic Lipschitz condition, i.e., there exist two positive $\mathbb{H}$-predictable processes $\gamma, \bar{\gamma}$ such that on $\mathbb{M}$
\begin{equation}\label{eqn;stochlip}
\bigg| f(t, \omega, y, \theta  ( \cdot ) ) - f(t, \omega, y', \theta' (\cdot ) ) \bigg| ^2
	\leq \gamma_t (\omega ) |y - y'|^2
	+\bar{\gamma}_t (\omega ) \left( ||| \theta ( \cdot ) - \theta' ( \cdot ) |||_t (\omega ) \right)^2 ,
\end{equation}
where:
\begin{equation}
\left( ||| \theta ( \cdot ) |||_t (\omega ) \right)^2 = \int_0^{+\infty} \theta^2(z) K_t^{\omega}(dz) \ge0 .
\end{equation}
Exploiting the definition of $f$ in Equation \eqref{eq:BSDE_gen}, we need first of all to deal with the $\esssup$:
\[
\bigl| f(t, \omega, y, \theta  ( \cdot ) ) - f(t, \omega, y', \theta' (\cdot )) \bigr| ^2
	\leq \left( \esssup_{u \in U} \bigl| \widetilde{f} ( t, \omega, y, \theta  ( \cdot ), u ) -  \widetilde{f} ( t, \omega, y', \theta'  ( \cdot ), u ) \bigr| \right)^2,
\]
and we preliminarly work on the absolute value difference involving $\widetilde{f}$:
\[
\begin{split}
& \bigl| \widetilde{f} ( t, \omega, y, \theta  ( \cdot ), u ) -  \widetilde{f} ( t, \omega, y', \theta'  ( \cdot ), u ) \bigr| \\
	&= \left| (y- y') \eta e^{r (T-t)} q^u_t(\omega) + \int_{0}^{+\infty} \left( y- y' + \theta(z) - \theta'(z)\right) \left( e^{-\eta e^{r(T-t)}(z-\Phi(z,u))} - 1 \right) K_t^{\omega}(dz) \right| \\
	&\le \bigl| y- y' \bigr| \eta e^{r (T-t)} q^{u_M}_t (\omega)
	+  \int_{0}^{+\infty} \left| y- y' \right| K_t^{\omega} (dz)
	+  \int_{0}^{+\infty} \left| \theta(z) - \theta'(z) \right| K_t^{\omega} (dz) \\
\end{split}
\]
where we have used the boundedness of $| e^{-\eta e^{R(T-t)}(z-\Phi(z,u))} - 1 |$ and that  $q_t^u \leq q^{u_M}_t$ for any $u \in U$. Now, since the inequality above does not depend on $u$ we also have that the $\esssup_{u \in U}$  satisfies it and we can take its square  (we use here the trivial relation $ (a+ b+c)^2 \leq 3 (a^2 + b^2 + c^2)$), finding:
\[
\begin{split}
& \left( \esssup_{u \in U} \bigl| \widetilde{f} ( t, \omega, y, \theta  ( \cdot ), u ) -  \widetilde{f} ( t, \omega, y', \theta'  ( \cdot ), u ) \bigr| \right)^2 \\
	&\le 3 \bigl| y- y' \bigr|^2 \eta^2 e^{2r (T-t)} (q^{u_M}_t (\omega))^2
	+ 3\left( \int_{0}^{+\infty} \left| y- y' \right| K_t^{\omega} (dz) \right)^2
	+ 3 \left( \int_{0}^{+\infty} \left| \theta(z) - \theta'(z) \right| K_t^{\omega} (dz) \right)^2. \\
\end{split}
\]
Recalling now that the transition kernel reads $K_t^\omega(dz) = \pi_{t^-} (\lambda) F^\a(\ud z)$ we use the following, for an integrable function $\vartheta$:
\begin{equation*}
\left( \int_0^{+ \infty} |\vartheta(\omega , z)| \pi_{t^-} (\lambda) F^\a(\ud z) \right)^2 \leq \int_0^{+\infty} |\vartheta(\omega , z)|^2 \pi_{t^-}^2 (\lambda) F^\a(\ud z) \cdot \underbrace{ \int_0^{+ \infty}  F^\a(\ud z)}_{=1} .
\end{equation*}
So, we find:
\[
\begin{split}
& \left( \esssup_{u \in U} \bigl| \widetilde{f} ( t, \omega, y, \theta  ( \cdot ), u ) -  \widetilde{f} ( t, \omega, y', \theta'  ( \cdot ), u ) \bigr| \right)^2 \\
	&\le 3\bigl| y- y' \bigr|^2 \eta^2 e^{2r (T-t)} (q^{u_M}_t (\omega))^2
	+ 3 \left| y- y' \right|^2 \pi_{t^-}^2 (\lambda)
	+ 3 \int_{0}^{+\infty} \left| \theta(z) - \theta'(z) \right|^2 \pi_{t^-} (\lambda) K_t^\omega(dz) \\
	&= 3\bigl| y- y' \bigr|^2 \left( \eta^2 e^{2r (T-t)} \left(q^{u_M}_t (\omega) \right)^2 +  \pi_{t^-}^2 (\lambda) \right)
	+ 3 \int_{0}^{+ \infty} \left| \theta(z) - \theta'(z) \right|^2 \pi_{t^-} (\lambda) K_t^\omega(dz) .
\end{split}
\]
So, the target, being Equation \eqref{eqn;stochlip}, is reached and we have the following values for the stochastic Lipschitz coefficients $\gamma_t$ and $\bar \gamma_t$:
\begin{align*}
\gamma_t &= 3 \eta^2 e^{2r (T-t)} ( q^{u_M}_t )^2 + 3 \pi_{t^-}^2 (\lambda) , \\
\bar{\gamma}_t &= 3 \pi_{t^-} (\lambda),
\end{align*}
which, as expected, are independent of the control $u$.
\item[$\mathbf{(F4)}$] Since by definition $ \alpha _{\cdot} ^2 = \max \{ \sqrt{\gamma_\cdot}, \bar \gamma_{\cdot} \}$, here we find:
$$
\alpha_s^2 =  \max \left\{ \sqrt{3 \eta^2 e^{2r (T-s)} ( q^{u_M}_s )^2 + 3 \pi_{s^-}^2 (\lambda)}, 3\pi_{s^-} (\lambda)\right\}
$$
and also $A_t = \int_{0}^{t} \alpha _{s}^2 \,ds $, so that we can easily verify that the inequality $\Delta A_t \leq \Phi, \P-$a.s. holds true for any $\Phi >0$ since $A$ has no jumps.
Notice that $\mathbf{(F2)}$ requires that the terminal condition $\xi = e^{- \eta X^T_N}$ belongs to the set of  $\mathcal H_T-$measurable random variables such that $\mathbb E \left[ e^{\widehat \beta A_T} e^{- 2 \eta X^N_T} \right] < \infty$, for some $\widehat \beta > 0$.
This is true for any $\widehat \beta >0$, since $\alpha_s^2 \le \sqrt 3 \eta e^{r (T-s)}  q^{u_M}_s  + 3 \pi_{s^-} (\lambda)$ and so
\begin{eqnarray*}
\mathbb E \left[ e^{\widehat \beta A_T} e^{- 2 \eta X^N_T} \right] &\le& \mathbb E \left[ e^{\widehat \beta \sqrt 3 \eta  \int_{0}^{T}  e^{r (T-s)}  q^{u_M}_s ds} e^{ 3 \widehat \beta \int_0^T  \pi_{s^-} (\lambda) ds } e^{- 2 \eta X^N_T} \right],
\end{eqnarray*}
which is finite for any $\widehat \beta >0$ thanks to Assumption \ref{ass_app_premium} $(ii)$ (see also Lemma \ref{lem:finite}).
\item[$\mathbf{(F5)}$] Finally, by using the same $\widehat\beta >0$ and $A$ introduced to prove $\mathbf{(F4)}$, we find:
\begin{equation}
\mathbb{E} \left[  \int_0^T e^{\widehat\beta A_t} \frac{|f (t, 0, 0, 0) |^2}{\alpha_t^2}  dt \right] < \infty,
\end{equation}
since here $f (t,0, 0, 0) = - \esssup_{u \in \mathcal{U}} \widetilde{f} (t, 0, 0, u_t) = 0$.
\end{itemize}
It now remains to prove that the quantity
$$
M^{\Phi}(\widehat \beta) = \frac{9}{\widehat\beta} + \frac{\Phi^2 (2 + 9 \widehat \beta)}{\sqrt{\widehat \beta^2 \Phi^2 +4} -2} \exp{\left( \frac{\widehat \beta \Phi + 2 -\sqrt{\widehat \beta^2 \Phi^2 +4}}{2} \right)}
$$
with $\Phi>0$ introduced in $\mathbf{(F4)}$ and $\widehat \beta >0$, satisfies $M^{\Phi}(\widehat \beta) < \frac12$.
Thanks to Papapantoleon, Possamai and Saplaouras \cite[Lemma 3.4]{Papa_Possa_Sapla2018}, for $\widehat \beta$ sufficiently large, we know that since $\lim_{\widehat \beta \rightarrow \infty} M^\Phi(\widehat \beta) = 9 e \Phi$ then it suffices to take $\Phi < \frac{1}{18 e}$.

It remains to show that $(Y,\Theta^{Y}) \in \mathcal{L}^2 \times \widehat{\mathcal{L}}^2$. According to Papapantoleon, Possamai and Saplaouras \cite[Theorem 3.5]{Papa_Possa_Sapla2018}, we know that
\[
\mathbb{E}\left[ \int_0^T e^{\widehat\beta A_t} \alpha_t^2 \abs{Y_t}^2 \,dt\right] <+\infty
\]
and we also notice that $\alpha^2_t \ge 3 \pi_{t^-} (\lambda) \ge 3 \min\{ \lambda_0, \beta \}$ and this implies $ \mathbb{E}\left[ \int_0^T e^{\widehat\beta A_t} \abs{Y_t}^2 \,dt\right] <+\infty$ and therefore $Y \in \mathcal{L}^2$. The same argument applies to prove that $\Theta^{Y}\in\widehat{\mathcal{L}}^2$.
\end{proof}

 \end{document}